\documentclass[11pt]{amsart}

\usepackage[colorlinks=true, pdfstartview=FitV, linkcolor=blue, 
citecolor=blue]{hyperref}

\usepackage{amssymb,amsmath,accents}
\usepackage{bbm}
\usepackage{graphicx}
\usepackage{a4wide}

\newtheorem{theorem}{Theorem}[section]
\newtheorem{prop}[theorem]{Proposition}
\newtheorem{lemma}[theorem]{Lemma}     
\newtheorem{fact}[theorem]{Fact}
\newtheorem{coro}[theorem]{Corollary}
\theoremstyle{definition}
\newtheorem{remark}[theorem]{Remark}

\newcommand{\ts}{\hspace{0.5pt}}
\newcommand{\nts}{\hspace{-0.5pt}}
\newcommand{\RR}{\mathbb{R}\ts}
\newcommand{\CC}{\mathbb{C}}
\newcommand{\ZZ}{\mathbb{Z}\ts}
\newcommand{\NN}{\mathbb{N}}
\newcommand{\QQ}{\mathbb{Q}\ts}
\newcommand{\XX}{\mathbb{X}}
\newcommand{\YY}{\mathbb{Y}}

\newcommand{\cA}{\mathcal{A}}
\newcommand{\cB}{\mathcal{B}}
\newcommand{\cE}{\mathcal{E}}
\newcommand{\cK}{\mathcal{K}}
\newcommand{\cL}{\mathcal{L}}
\newcommand{\cM}{\mathcal{M}}

\newcommand{\vL}{\varLambda}
\newcommand{\mf}{\mathfrak{m}^{}_{\mathrm{c}}}
\newcommand{\mc}{m^{}_{\mathrm{c}}}

\newcommand{\dd}{\,\mathrm{d}}
\newcommand{\ee}{\ts\mathrm{e}}
\newcommand{\ii}{\mathrm{i}\ts}

\newcommand{\one}{\mathbbm{1}}
\newcommand{\Mat}{\mathrm{Mat}}
\newcommand{\GL}{\mathrm{GL}}

\newcommand{\diag}{\ts\ts\mathrm{diag}\ts}
\newcommand{\exend}{\hfill$\Diamond$}

\newcommand{\myfrac}[2]{\frac{\raisebox{-2pt}{$#1$}}
      {\raisebox{0.5pt}{$#2$}}}

\DeclareMathOperator{\card}{card}
\DeclareMathOperator{\dens}{dens}
\DeclareMathOperator{\sinc}{sinc}
\DeclareMathOperator{\vol}{vol}
\DeclareMathOperator{\real}{Re}
\DeclareMathOperator{\imag}{Im}
\DeclareMathOperator{\im}{im}
\DeclareMathOperator{\tr}{tr}
\DeclareMathOperator{\supp}{supp}

\DeclareFontFamily{U}{mathx}{\hyphenchar\font45}
\DeclareFontShape{U}{mathx}{m}{n}{ <5> <6> <7> <8> <9> <10>
   <10.95> <12> <14.4> <17.28> <20.74> <24.88> mathx10 }{}
\DeclareSymbolFont{mathx}{U}{mathx}{m}{n}
\DeclareFontSubstitution{U}{mathx}{m}{n}
\DeclareMathAccent{\widecheck}{0}{mathx}{"71}

\newcommand{\sB}{\ts\underline{\nts B\!}\,}

\newcommand{\defeq}{\mathrel{\mathop:}=}

\newcommand{\oplam}{\mbox{\Large $\curlywedge$}}

\begin{document}

\title[Rauzy fractals and point spectrum of 1D inflation
tilings]{Fourier transform of Rauzy fractals and \\[2mm]
  point spectrum of 1D Pisot inflation tilings}
\author{Michael Baake}
\address{Fakult\"at f\"ur Mathematik, Universit\"at
  Bielefeld, \newline
  \indent Postfach 100131, 33501 Bielefeld,
  Germany}
\email{mbaake@math.uni-bielefeld.de}

\author{Uwe Grimm}
\address{School of Mathematics and Statistics,
  The Open University,\newline 
  \indent Walton Hall, Milton Keynes MK7 6AA, UK} 
\email{uwe.grimm@open.ac.uk}

\begin{abstract}
  Primitive inflation tilings of the real line with finitely many
  tiles of natural length and a Pisot--Vijayaraghavan unit as
  inflation factor are considered. We present an approach to the pure
  point part of their diffraction spectrum on the basis of a Fourier
  matrix cocycle in internal space. This cocycle leads to a transfer
  matrix equation and thus to a closed expression of matrix Riesz
  product type for the Fourier transforms of the windows for the
  covering model sets. In general, these windows are complicated Rauzy
  fractals and thus difficult to handle.  Equivalently, this approach
  permits a construction of the (always continuously representable)
  eigenfunctions for the translation dynamical system induced by the
  inflation rule.  We review and further develop the underlying
  theory, and illustrate it with the family of Pisa substitutions,
  with special emphasis on the Tribonacci case.
\end{abstract}

\maketitle

\section{Introduction}

Inflation tilings of the real line with an inflation (or stretching)
factor $\lambda$ that is a Pisot--Vijayaraghavan (PV) number are
intimately related to cut and project sets. In the best case, which is
the topic of the famous Pisot substitution conjecture
\cite{Bernd,Aki}, their vertex points (in the geometric realisation
with intervals of natural length) are regular model sets themselves,
and thus have pure point spectrum, equivalently in the dynamical or in
the diffraction sense \cite{LMS,BL,BL-review}. More generally, they
might have mixed spectrum, see \cite{BG,BGM} and references therein
for examples, but the PV-nature of $\lambda$ still implies that they
lead to Meyer sets and thus have non-trivial point spectrum
\cite[Sec.~5.10]{Nicu}; see also \cite{Str}.

When analysing such inflation tilings, one quickly encounters covering
model sets with complicated windows, known as Rauzy fractals
\cite{Rauzy,PFBook}, which are compact sets of positive measure that
are topologically regular (that is, they are the closure of their
interior) and perfect (that is, they have no isolated points), but
display a fractal boundary and often also a non-trivial fundamental
group.  While a lot is known about Rauzy fractals, see
\cite{ST,Bernd,PFBook} and references therein, it is not obvious how
to calculate their Fourier transform in closed form, which is needed
to determine the diffraction intensities of the tiling system
explicitly.  Phrased differently, but equivalently, this Fourier
transform is also needed to calculate the eigenfunctions of the
corresponding dynamical system under the translation action of $\RR$;
compare \cite{Daniel,BL-review}.

The purpose of this contribution is to reconsider this problem in a
constructive and explicit way. In particular, our goal is to make
the Fourier--Bohr (FB) coefficients or amplitudes (and thus also the
eigenfunctions) of such inflation tilings available, via a quadratic
form with a matrix that can be expressed as an infinite matrix Riesz
product. Since the latter turns out to be compactly and rapidly
converging, all quantities are efficiently computable.  Here, we solve
the problem for inflation tilings of the real line with finitely many
prototiles and an inflation factor that is a PV unit. The extension to
general PV numbers and to higher dimensions will be treated
separately, as this requires a bigger machinery, algebraically
and analytically. \smallskip

The paper is organised as follows. We begin by recalling the setting
of inflation tilings of the real line in Section~\ref{sec:infl}. Then,
in Section~\ref{sec:Minkowski}, we introduce the Minkowski embedding
and the description of our tilings (and point sets) in internal space,
which leads to a contractive iterated function system for the windows
of the covering model sets. This is followed by the introduction and
analysis of an internal cocycle in Section~\ref{sec:cocycle}, which
leads to a matrix Riesz product expression for the Fourier transform
of the Rauzy windows, and thus also for the spectral quantities we are
after. In this context, in Section~\ref{sec:FB}, we establish an
important connection between the FB coefficients of PV
inflation point sets and those of the covering model sets, which
emerges through a specific uniform distribution result.  Then, in
Section~\ref{sec:examples}, we embark on a number of illustrative
examples from the family of Pisa substitutions (including some based
on cubic and quartic number fields), followed by an example of
covering degree $2$ in Section~\ref{sec:twisted} and a brief outlook.

\section{Inflation tilings of the real line}\label{sec:infl}

Let us begin with the symbolic side of the problem, where we consider
a \emph{primitive substitution} $\varrho$ on a finite alphabet
$\cA = \{ a^{}_{1}, \ldots , a^{}_{N}\}$.  Here, the mapping
$a^{}_{i} \mapsto \varrho (a^{}_{i})$ is usually specified by the
$N$-tuple
$\bigl( \varrho (a^{}_{1}), \ldots , \varrho (a^{}_{N}) \bigr)$. The
\emph{substitution matrix} of $\varrho$ is $M$, where $M_{ij}$ counts
the number of letters of type $a^{}_{i}$ in $\varrho (a^{}_{j})$; see
\cite{PFBook,Q,TAO} for general background and results.  We denote the
characteristic polynomial of $M$ by $p (x)$, which is monic, but need
not be irreducible over $\ZZ$ in our setting, meaning that we can also
include a variety of systems with mixed spectrum.

Let $\lambda = \lambda^{}_{\mathrm{PF}}$ be the 
\emph{Perron--Frobenius} (PF)
eigenvalue of $M$. As $M$ is primitive by assumption, we know
\cite[Thm.~8.4.4]{HJ} that there are strictly positive left and right
eigenvectors for $\lambda$, denoted\footnote{Since we will be using
  left and right eigenvectors throughout, we adopt Dirac's bra-c-ket
  notation, where $\langle u \ts | \ts v \rangle$ then stands for the
  sesquilinear inner product in $\CC^{N}$, which becomes bilinear when
  restricted to $\RR^{N}$.}  by $\langle u \ts |$ and $|\ts v
\rangle$, which we assume to be normalised such that
\[
    \langle 1 | \ts v \rangle \, = \, \langle u \ts | \ts v \rangle
    \, = \, 1 \ts .
\]
Here, $\langle 1 | \defeq \langle 1, \ldots, 1 |$ is the row vector
with $N$ equal entries $1$. In the substitution context, the first
condition thus ensures that the entries of $|\ts v \rangle$ encode the
relative letter frequencies in the symbolic sequences defined by
$\varrho$, while the second condition implies that
\begin{equation}\label{eq:def-P}
     P \, \defeq \, |\ts v \rangle \langle u \ts |
\end{equation}
is a projector of rank $1$, so $P^2 = P$ with $P(\RR^{N}) = \im (P) =
\RR \ts |\ts v\rangle$, where all entries of $P \in \Mat(N,\RR)$ are
strictly positive. Now, let $\|.\|$ be any matrix norm, not
necessarily a sub-multiplicative one, where we recall that all matrix
norms are equivalent here.  Then, the following property is
standard; see \cite[Thm.~8.5.1]{HJ}.

\begin{fact}\label{fact:projector}
  For a primitive, non-negative matrix\/ $M \in \Mat (N,\RR)$ with PF
  eigenvalue\/ $\lambda$, one has\/
  $\lim_{n\to\infty} \lambda^{-n} M^n = P$, where\/ $P$ is the projector
  from\/ \eqref{eq:def-P}. This convergence also entails that\/
  $0 < \sup_{n\in\NN} \| \lambda^{-n} M^n \| < \infty$. \qed
\end{fact}

Working with PV substitutions, we may as well profit from the
underlying \emph{geometry} by turning the symbolic sequences into tilings;
see \cite{Boris,TAO} for general background and \cite{CS1,CS2} for the
justification why this does not change the spectral type of our
system. Here, we choose intervals of natural length, meaning
proportional to the entries of $\langle u \ts |$, with control points
on their left endpoints. As all entries $u^{}_{i}$ of $\langle u\ts |$
lie in $\QQ (\lambda)$, one normally multiplies them with their common
denominator, so that they become elements of $\ZZ [\lambda]$, but not
of a proper, $\lambda$-invariant submodule.  For
each sequence in the symbolic hull defined by $\varrho$, this leads to
a multi-component or \emph{typed} point set,
$\vL = \dot{\bigcup}_{i} \vL_{i}$, where the $\vL_i$ emerge from the $N$
distinct types of control points and now form pairwise disjoint
subsets of $\ZZ[\lambda]$.

\begin{remark}\label{rem:geo}
  Let us mention one consequence of the geometric setting.  When
  $\ell_i = \alpha\ts u_i$ with $1\leqslant i \leqslant N$ are the
  chosen interval lengths, the average distance between neighbouring
  control points in $\vL$ is well defined, compare
  \cite[Sec.~4.3]{TAO}, and reads
\[
  \bar{\ell} \, =   \sum_{i=1}^{N} v_i \ts\ts \ell_i
  \, = \, \alpha \ts \langle u \ts | \ts v \rangle
  \, = \, \alpha \ts ,
\]
so we get $\dens(\vL) = 1/\alpha$ as the density of $\vL$, and
$\dens (\vL^{}_{i}) = v^{}_{i} \ts \dens (\vL)$ for
$1 \leqslant i \leqslant N$.  \exend
\end{remark}

Next, we invoke the set-valued \emph{displacement matrix}
$T = (T_{ij})^{}_{1 \leqslant i,j \leqslant N}$, where the set
$T_{ij}$ consists of all relative (geometric) positions of tiles of
type $a^{}_i$ in the supertile\footnote{Note that, by slight abuse of
  notation, we use $\varrho$ both for the symbolic substitution and
  for the geometric inflation, where the meaning will always be clear
  from the context.}  $\varrho (a^{}_j)$; see \cite{BFGR,BG,BGM} for
background. This gives rise to the \emph{Fourier matrix} of $\varrho$
via $B \defeq \widecheck{\delta^{}_{T}}$, so
\[
    B_{ij} (k) \, = \sum_{x\in T_{ij}} \ee^{2 \pi \ii x k} ,
    \quad k \in \RR \ts ,
\]
which is a trigonometric polynomial. Since $\card (T_{ij}) = M_{ij}$,
one has the inequality
\begin{equation}\label{eq:max-B}
    \lvert B_{ij} (k) \rvert \, \leqslant \, M_{ij}
\end{equation}
for all $i,j$ and all $k\in\RR$, where we get equality for
$k=0$ because $B (0) = M$.

Given $B$, we construct a \emph{cocycle} (in the sense of
\cite[Sec.~2.1]{BP}, over the dilation $k\mapsto \lambda k$) from the
Fourier matrix, via
\begin{equation}\label{eq:def-cocycle}
   B^{(n)} (k) \, \defeq \, B(k) B(\lambda k) \cdots 
   B( \lambda^{n-1} k ) \ts ,
\end{equation}
so $B^{(1)} (k) = B (k)$ together with
\begin{equation}\label{eq:cocycle1}
  B^{(n+1)} (k) \, =\,  B^{(n)} (k) B( \lambda^n k) 
  \, = \, B(k) B^{(n)} (\lambda k) 
\end{equation}
for $n\geqslant 1$.  Inductively, one can check that $B^{(n)} (k)$ is
the Fourier matrix of $\varrho^n$, see \cite[Fact~3.6]{BGM}, with
$B^{(n)} (0) = M^n$. Further, one has $B^{(n+m)} (k) =
B^{(n)} (k) B^{(m)} (\lambda^n k)$ for all $n,m\geqslant 0$, with
the convention $B^{(0)} \defeq \one$.

Recall that a matrix norm $\|.\|$ is called \emph{weakly monotone}
when $\| A \| \leqslant \big\| \ts \lvert A \rvert \ts \big\|$ holds
for all $A \in \Mat (N,\CC)$, where $\lvert A \rvert$ denotes the
matrix with entries $\lvert A_{ij} \rvert$; see \cite{JN} for a
general exposition of monotonicity properties of vector and matrix
norms.  With this and Eq.~\eqref{eq:max-B}, the following property is
immediate from Fact~\ref{fact:projector}.

\begin{fact}\label{fact:B-mat}
  The entries of the cocycle~\eqref{eq:cocycle1} satisfy\/
  $\lvert B^{(n)}_{ij} (k) \rvert \leqslant (M^n)^{}_{ij}$, for all\/
  $i,j$ and all\/ $k\in\RR$.  Consequently, if\/ $\| . \|$ is any
  weakly monotone matrix norm, $\| \lambda^{-n} B^{(n)} (k) \|$ is
  uniformly bounded on\/ $\RR$, which means that
\[
    c^{}_{B} \, \defeq \, \sup_{n\in\NN} \,\sup_{k\in\RR}
    \| \lambda^{-n} B^{(n)} (k) \|
\]   
is finite, with\/ $0 < c^{}_{B} < \infty$.  Moreover, for
all\/ $i,j$ and all\/ $k\in\RR$, one has
\[
  0 \, \leqslant \, \liminf_{n\to\infty} \lambda^{-n}
  \lvert B^{(n)}_{ij} (k) \rvert \, \leqslant \, \limsup_{n\to\infty}
  \lambda^{-n} \lvert B^{(n)}_{ij} (k) \rvert \, \leqslant \,
  P^{}_{ij} \ts ,
\]
where the\/ $P^{}_{ij}$ are the matrix elements of the projector\/
$P$ from Eq.~\eqref{eq:def-P}. This leads to more specific results
on\/ $c^{}_{B}$ depending on the matrix norm chosen.
\qed
\end{fact}

Let us from now on assume that $\lambda$ is a PV \emph{unit} of degree
$d\leqslant N$.  Since $M$ is an integer matrix and the equations for
the left and right eigenvectors to $\lambda$ can thus be solved in the
field $\QQ (\lambda)$, a natural object to consider is the
$\ZZ$-module $L\defeq\ZZ[\lambda]=\langle 1,\lambda, \ldots ,
\lambda^{d-1} \rangle^{}_{\ZZ}$ of rank $d$, which satisfies $\lambda
L = L$. This is the main reason to choose the interval lengths
$(\ell^{}_{1}, \ldots , \ell^{}_{\nts N})$ for the tiling such that
$\ZZ [\lambda]$ comprises all possible coordinates of our control
points (relative to one of them, which may be placed at $0$ without
loss of generality). We assume that $L$ is optimal relative to the
control point set in the sense that no proper, $\lambda$-invariant
submodule of $L$ comprises all of those points (otherwise, we change
the natural interval lengths so that this is true). Next, we extract
and harvest some intrinsic geometric information from this setting.

\section{Minkowski embedding and internal 
space}\label{sec:Minkowski}

Let us recall the \emph{Minkowski embedding} 
from \cite[Sec.~3.4]{TAO}, now
tailored to $L = \ZZ[\lambda]$. Since $\lambda$ has degree $d$, this
will lead to a lattice $\cL \subset \RR^d$ as follows. There are $r$ real
algebraic conjugates of $\lambda$, and $s$ complex conjugate pairs, so
$d = r + 2 s$, which are defined via the irreducible, monic polynomial
in $\ZZ[x]$ that has $\lambda$ as a root. This polynomial is a factor
of the characteristic polynomial $p$ of $M$ in our case. Consequently,
there are $r\geqslant 1$ real field isomorphisms
$\kappa^{}_{1}, \ldots , \kappa^{}_{r}$, with
$\kappa^{}_{1} = \mathrm{id}$, and $s\geqslant 0$ complex field
isomorphisms $\sigma^{}_{1}, \ldots , \sigma^{}_{\nts s}$, together
with their complex conjugates,
$\overline{\sigma^{}_{1}}, \ldots , \overline{\sigma^{}_{\nts s}}$.  In this
setting, we can define a $\ZZ$-linear mapping
$\varPhi \! : \, \ZZ [\lambda] \xrightarrow{\quad}\RR^d$ by
\[
  x \, \mapsto \, \bigl(x, \kappa^{}_{2} (x),
  \ldots , \kappa^{}_{r} (x), \real(\sigma^{}_{1} (x)),
  \imag(\sigma^{}_{1} (x)), \ldots,
  \real(\sigma^{}_{\nts s} (x)),
  \imag(\sigma^{}_{\nts s} (x))\bigr) ,
\]
which extends to a $\QQ$-linear mapping on the field $\QQ (\lambda)$.
Clearly, each image point is of the form $\varPhi(x)=(x,x^{\star})$
with $x^{\star}\in\RR^{d-1}$. The induced map
$\star\! :\, \QQ (\lambda) \xrightarrow{\quad}\RR^{d-1}$ is called the
\emph{$\star$-map} of the underlying cut and project scheme; compare
\cite[Sec.~7.2]{TAO} or \cite{Bob}. Often, only its restriction to 
$\ZZ [\lambda]$ is named the $\star$-map, but it has a unique extension
to $\QQ (\lambda)$, as used here.

It is a standard result of algebraic number theory
\cite[Sec.~I.5]{Neukirch} that
$\cL\defeq \varPhi (L) = \varPhi \bigl(\ZZ[\lambda]\bigr)$ is indeed a
lattice in $\RR^d$, where we prefer the real version over the
(algebraically) perhaps more natural one with
$\RR^r \! \times \nts\nts \CC^s$ because we will need Fourier
transforms shortly. So, we obtain the following Euclidean \emph{cut
  and project scheme}, or CPS for short; see \cite[Sec.~7.2]{TAO} and
references therein for more.
\begin{equation}\label{eq:CPS}
\renewcommand{\arraystretch}{1.2}\begin{array}{r@{}ccccc@{}l}
   & \RR & \xleftarrow{\;\;\; \pi \;\;\; } 
      & \RR \nts\nts \times \nts\nts \RR^{d-1} & 
        \xrightarrow{\;\: \pi^{}_{\text{int}} \;\: } & \RR^{d-1} & \\
   & \cup & & \cup & & \cup & \hspace*{-1ex} 
   \raisebox{1pt}{\text{\footnotesize dense}} \\
   & \pi (\cL) & \xleftarrow{\;\ts 1-1 \;\ts } & \cL & 
        \xrightarrow{ \qquad } &\pi^{}_{\text{int}} (\cL) & \\
   & \| & & & & \| & \\
   & L & \multicolumn{3}{c}{\xrightarrow{\qquad\qquad\quad\star
       \quad\qquad\qquad}} 
       &  {L_{}}^{\star\nts}  & \\
\end{array}\renewcommand{\arraystretch}{1}
\end{equation}
Here, $\pi$ and $\pi^{}_{\text{int}}$ denote the canonical projections.
Such a CPS is abbreviated as $(\RR, \RR^{d-1}, \cL)$.

To continue, observing that
$\RR^d = \RR \nts\nts \times \nts\nts \RR^{d-1}$, we also need the
linear mapping $Q$ on $\RR^{d-1}$ that is induced by the dilation
$x \mapsto \lambda x$ (acting on the first component) in internal
space. It is immediate from the structure of $\varPhi$ and the CPS
\eqref{eq:CPS} that we get
\begin{equation}\label{eq:def-Q}
  Q \, = \, \diag \bigl(\kappa^{}_{2} (\lambda), \ldots,
  \kappa^{}_{r} (\lambda) \bigr) \oplus \,
  \bigoplus_{i=1}^{s}\begin{pmatrix}
  \real(\sigma^{}_{i}(\lambda)) & - \imag(\sigma^{}_{i}(\lambda)) \\
  \imag(\sigma^{}_{i}(\lambda)) & \real(\sigma^{}_{i}(\lambda))
  \end{pmatrix} .
\end{equation}
Clearly, $Q\in\Mat (d\! - \! 1,\RR)$ is a normal matrix (so,
$[Q^T,Q]=0$) and a contraction, the latter because $\lambda$ is a PV
number, so all its algebraic conjugates lie strictly inside the unit
disk.

Since $\{ 1, \lambda, \ldots , \lambda^{d-1} \}$ is a $\ZZ$-basis of
$\ZZ[\lambda]$, a basis matrix $\cB$ for $\cL$ can be chosen from
here, with columns $\varPhi(\lambda^i)^T$ for
$0\leqslant i \leqslant d-1$.  The \emph{dual matrix},
$\cB^* \defeq (\cB^{-1})^T$, is then a basis matrix of the dual
lattice,
\[
  \cL^* \, \defeq \, \{y\in\RR^{d}: \langle
  x \ts | \ts y\rangle\in\ZZ \text{ for all } x\in\cL\}\ts . 
\]
Clearly, $\dens (\cL^{*}) = \dens (\cL)^{-1} = \lvert  \ts \det
(\cB)\rvert$.

\begin{remark}\label{rem:spec}
  When $L=\ZZ[\lambda]$ is minimal in the sense that the underlying
  point set is contained in $L$, but not in any $\lambda$-invariant
  submodule of it, the corresponding \emph{Fourier module} of the
  point set or tiling can be extracted from the first row of the dual
  basis matrix $\cB^*$ as
\[
  L^{\circledast} \, = \, \big\langle \cB^{*}_{1i}: 1\leqslant
  i\leqslant d\ts \big\rangle_{\ZZ},
\]
which is the projection of $\cL^{*}$ to the first component.

There is an intrinsic way to define $L^{\circledast}$ as follows.
With the Galois isomorphisms $\kappa^{}_{i}$ and $\sigma^{}_{\nts j}$
from above, one defines the number-theoretic \emph{trace}
on $\QQ (\lambda)$ as
\begin{equation}\label{eq:trace}
     \tr (x) \, \defeq \ts \sum_{i=1}^{r} \kappa^{}_{i} (x) 
     \, + \sum_{j=1}^{s} \bigl( \sigma^{}_{\nts j} (x) \, +
     \overline{\sigma^{}_{\nts j}} (x) \bigr) \, = 
     \sum_{i=1}^{r} \kappa^{}_{i} (x) \, + 2
     \sum_{j=1}^{s} \real \bigl( \sigma^{}_{\nts j} (x) \bigr).
\end{equation}
Then, one has
$L^{\circledast} = \{ y \in \QQ (\lambda) : \tr (xy) \in \ZZ \text{
  for all } x \in L \}$, which bypasses the explicit embedding step
employed above, though it is of course equivalent to it.

In our setting, the Abelian group $L^{\circledast}$ is the pure point
part of the dynamical spectrum, in additive notation, for the tiling
dynamical system induced by the inflation rule, where the dynamics is
given by the translation action of $\RR$; see \cite{BL-review} for
background.  \exend
\end{remark}

To continue, in the spirit of \cite{LW}, see also \cite[Ch.~4]{TAO},
we return to the displacement matrix $T$ of $\varrho$. By definition,
when considering the $a_i$ as tiles with natural length $\ell_{i}$ and
left endpoint placed at $0$, one has the \emph{stone inflation} 
(compare \cite[p.~148]{TAO} for the concept)
\[
  \lambda \ts a^{}_i \: =  \bigcup_{1\leqslant j \leqslant N}
  \, \bigcup_{t\in T_{ji}}    t+ a^{}_{j} \ts .
\]
More importantly, $T$ enters the induced inflation action on the point
sets via the iteration
\begin{equation}\label{eq:set-inflation}
  \vL_{i}^{\prime} \: =  \dot{\bigcup_{1\leqslant j \leqslant N}}
   \lambda \vL^{}_{j} + T^{}_{ij} \ts ,
\end{equation}
with a suitable and admissible initial condition, such as the left
endpoints of a legal pair of intervals, one placed at $0$ and the
other at the fitting position to the left. Here, we use ${}'$ to denote
the image under one iteration step. Note that this iteration produces
the control point sets of the corresponding successive tile
inflations.  The dot indicates that the union on the right-hand side
of \eqref{eq:set-inflation} is disjoint, while $+$ stands for the
Minkowski sum of point sets; compare \cite[Sec.~2.1]{TAO}.

\begin{remark}\label{rem:cycle}
  Note that the iteration based on \eqref{eq:set-inflation}, viewed in
  the local topology, need not converge to a single typed point set
  $\vL = \bigcup_{i} \vL_{i}$.  However, via a simple application of
  Dirichlet's pigeon hole principle, one can show convergence
  to a finite cycle of such typed point sets, starting from a fixed,
  admissible (or legal) initial configuration. Each member of this
  cycle is equally well suited to define the (geometric) hull as an
  orbit closure under translations; compare \cite[Chs.~4 and 5]{TAO}
  for details.  \exend
\end{remark}

Under the $\star$-map, \eqref{eq:set-inflation} turns into an
iteration of $N$ finite (and hence closed) point sets in $\RR^{d-1}$,
and thus into an \emph{iterated function system} (IFS) on
$\bigl(\cK \RR^{d-1}\bigr)^{N}$, where $\cK \RR^{m}$ with
\mbox{$m\in\NN$} denotes the space of non-empty, compact subsets of
$\RR^{m}$, equipped with the Hausdorff (metric) topology; see
\cite{BM00} or \cite[Sec.~4.6]{Bernd} and references therein for
background. Here, the multiplication by $\lambda$ is replaced by the
action of the contraction $Q$ from \eqref{eq:def-Q}, giving the
fixed point equations
\begin{equation}\label{eq:win-IFS}
  W^{}_{\nts i} \: =  \bigcup_{1 \leqslant j \leqslant N}
  Q \ts W^{}_{\nts j} + T^{\star}_{ij} \: = 
  \bigcup_{1\leqslant j \leqslant N} \,
  \bigcup_{t\in T_{ij}} Q W^{}_{j} + t^{\star}
\end{equation}
for $1 \leqslant i \leqslant N$. In this step, since the $W_i$ are
compact sets in $\RR^{d-1}$, the union on the right-hand side need no
longer be disjoint.  By Banach's contraction principle, there is a
unique solution to the IFS \eqref{eq:win-IFS} within $\bigl(\cK
\RR^{d-1}\bigr)^N$; see \cite[Thm.~1.1 and Prop.~1.3]{BM00}.  It is a
well-known fact that the compact sets $W_{i}$ can be Rauzy fractals
with complicated topological structure and boundary \cite{ST,Bernd};
see \cite[Sec.~7.4]{PFBook} and references therein for
background. Note that Banach's contraction principle also gives us
that each set $\vL^{\star}_{i}$ lies dense in the set $W_i$, which
plays the role of a window for the CPS \eqref{eq:CPS}, as we shall
exploit shortly.

\begin{remark}\label{rem:dual}
Let us briefly mention that Eq.~\eqref{eq:win-IFS} gives rise to a
\emph{dual inflation}, via multiplying from the left by $Q^{-1}$, 
which is an expansive mapping. This results in
\[
    Q^{-1} W^{}_{\nts i} \, = \bigcup_{1 \leqslant j \leqslant N}
     \bigcup_{t\in T_{ij}} W^{}_{j} + Q^{-1} t^{\star} ,
\]
where $Q^{-1} t^{\star} = (t/\lambda)^{\star}$. Iterating this rule in
internal space either leads to a tiling or to a multiple cover of
internal space, where the covering degree is constant almost
everywhere. This follows from \cite[Cor.~5.81]{Bernd}, which extends
an earlier idea from \cite{ItoRao}.  \exend
\end{remark}

Next, we recall a well-known result about the solution of the IFS
\eqref{eq:win-IFS}, which can be seen as a special case of
\cite[Prop.~4.99]{Bernd}. Since it is of crucial importance to our
further arguments, we also include a proof that is tailored to our
setting. The latter considers primitive inflation tilings of $\RR$,
with a PV unit $\lambda$ as inflation factor and natural lengths of
the $N$ intervals chosen such that all control point positions lie in
$\ZZ[\lambda]$, but in no proper, $\lambda$-invariant submodule of it.

\begin{lemma}\label{lem:disjoint}
  Under our general assumptions, the solution\/
  $(W^{}_{1}, \ldots , W^{}_{\! N})$ to the IFS\/ \eqref{eq:win-IFS}
  is row-wise measure-disjoint, which means that, for each
  $1\leqslant i \leqslant N$, any two distinct sets on the right-hand
  side of\/ \eqref{eq:win-IFS} intersect at most in a Lebesgue-null
  set.

  Moreover, there is a number\/ $\eta>0$ such that\/
  $\vol (W_i) = \eta \ts v^{}_{i}$ holds for all\/
  $1\leqslant i \leqslant N$, where\/ $v^{}_{i}$ is the relative
  frequency of the tiles of type\/ $i$.
\end{lemma}

\begin{proof}
  All $W_i$ are compact sets, hence measurable, with
  $\vol (W^{}_{\nts i}) \geqslant 0$. Let $\vL_{i}$ be the set of
  control points of type $i$ for one of the typed point sets
  $\vL = \dot\bigcup_{i} \vL_{i}$ that emerge from the limit cycle of the
  iteration \eqref{eq:set-inflation}, as explained in
  Remark~\ref{rem:cycle}.  By construction, due to the properties of
  the $\star$-map, we then know that
\[
  \vL_{i} \,\subseteq\, \oplam (W_{i}) \,\defeq\, 
  \{ x \in L : x^{\star} \in W_{i} \}\ts , 
\]
where $\vL_{i}$ is linearly repetitive, with
  $\dens (\vL_{i}) = v_{i} \dens (\vL) > 0$. Consequently, by invoking
  \cite[Prop.~3.4]{HuRi}, which is an extension of the density result
  \cite[Thm.~1]{Martin1} for regular model sets to the more general
  setting of weak model sets, we know that
\[
  0 \, < \, \dens (\vL_{i}) \, \leqslant \,
  \underline{\dens} \bigl(\oplam(W_{i}) \bigr)
  \, \leqslant \, \dens (\cL) \ts \vol (W_{i}) \ts ,
\]
where $\underline{\dens}$ refers to the (always existing) lower
density of a point set. This estimate means that we have
$\vol (W_{i}) > 0$ for all $1 \leqslant i \leqslant N$.
   
Now, due to a potential overlap of sets on the right-hand side of
Eq.~\eqref{eq:win-IFS}, one has
\begin{equation}\label{eq:PF-estimate}
   \vol (W^{}_{\nts i}) \, \leqslant \, \lvert \det (Q) \rvert 
   \sum_{j=1}^{N} \card (T^{}_{ij})\ts \vol (W^{}_{\nts j}) \ts ,
   \quad \text{for } 1 \leqslant i \leqslant N .
\end{equation}
Observing $\lvert \det (Q) \rvert = \lambda^{-1}$ and
$\card(T^{}_{ij}) = M^{}_{ij}$, this amounts to the vector inequality
\[
     M |w\rangle \, \geqslant \, \lambda \ts |w \rangle \ts ,
\]  
where $|w\rangle$ denotes the vector with entries
$\vol (W^{}_{\nts i})$ and the inequality holds for each
component. Since $\lambda$ is the PF eigenvalue of $M$, which is
primitive, and all entries of $|w\rangle$ are positive, we see that
$|w\rangle$ is a positive multiple of the right PF eigenvector of $M$;
compare \cite[Thm.~1.1]{Seneta} and its proof, which we need not
repeat here.

This means we have equality in \eqref{eq:PF-estimate}, which implies
the first claim, while the second is a consequence of $|\ts w \rangle$
being proportional to $|\ts v \rangle$.
\end{proof}

All mappings that occur in our IFS \eqref{eq:win-IFS} are of the form
$x\mapsto Qx+u$, and hence homeomorphisms of $\RR^{d-1}$.  Invoking
parts (i) and (iii) of \cite[Prop.~4.99]{Bernd}, one obtains the
following improvement of Lemma~\ref{lem:disjoint}.

\begin{prop}\label{prop:disjoint}
  Let\/ $(W^{}_{1}, \ldots , W^{}_{\! N})$ be the unique solution to the
  contractive IFS\/ \eqref{eq:win-IFS}. Then, under our assumptions, each\/
  $W_i \subset \RR^{d-1}$ is a perfect, topologically regular set of
  positive Lebesgue measure. Moreover, each boundary\/ $\partial W_i$
  has Lebesgue measure $0$.  \qed
\end{prop}

In view of this result, all $\oplam(W_{i})$ are \emph{regular model
  sets} for the cut and project scheme $(\RR , \RR^{d-1} , \cL)$, see
\cite{Bob,TAO} for background, hence also $\oplam(W)$ with
$W\! =\bigcup_{i} W_i$. The corresponding dynamical system
$(\XX,\RR)$, where $\XX$ is the orbit closure of $\oplam (W)$ in the
local topology and $\RR$ acts by translation, has pure point spectrum,
both in the diffraction and in the dynamical sense; see
\cite{TAO,BL,BL-review} and references therein. The same property
holds for the systems built from the translation orbit closure of any
of the $\oplam(W_{i})$.

\begin{remark}\label{rem:model}
  If we consider the weighted Dirac comb $\omega = \sum_{i}h_{i}\,
  \delta^{}_{\!\vL_{i}}$ with $h_{i}\in\CC$, the Bombieri{\ts}--Taylor
  (or consistent phase) property for primitive 
  inflation rules \cite[Thm.~3.23 and
    Rem.~3.24]{BGM} holds, which is an extension of the results of
  \cite{Daniel} to the typed point sets emerging from a primitive
  inflation rule. This implies the existence of coefficients $A_{i}
  (k)$, called scattering or diffraction \emph{amplitudes}, such that
  $\omega$ has the \emph{diffraction measure}
\begin{equation}\label{eq:intens}
    \widehat{\gamma^{}_{\omega}}\, = 
    \sum_{k\in L^{\circledast}} I(k)\, \delta^{}_{k}\quad\text{with}\quad
    I(k) \, = \, \Bigl| \textstyle{\sum\limits_{i}}\,
    h_{i} \, A_{i}(k) \Bigr|^{2} ,
\end{equation}
where the Fourier module $L^{\circledast}$ is the projection of 
$\cL^{*}$ into $\RR$ as in Remark~\ref{rem:spec}.

Now, assume in addition that $\oplam(W^{\circ}_{i}) \subseteq
\vL^{}_{i} \subseteq \oplam (W^{}_{i})$ holds for all $1\leqslant i
\leqslant N$, with the $W_i$ from the solution of
\eqref{eq:win-IFS}. Then, due to the inflation origin, our typed point
set $\vL = \dot{\bigcup}_{i} \vL_{i}$ consists of disjoint, regular
model sets. Consequently, by the general theory of model sets
\cite{Bob,TAO}, the amplitudes $A_{i} (k)$ for $k\in L^{\circledast}$
are given by
\begin{equation}\label{eq:ampli1}
     A_{i}(k) \, = \, \frac{\dens(\vL_{i})}{\vol(W_i)} \, 
     \widecheck{1^{}_{W_i}}(k^{\star}) \, = \, 
     \frac{\dens(\vL)}{\vol(W)} \, 
     \widecheck{1^{}_{W_i}}(k^{\star}) \ts ,
\end{equation}
where $1^{}_{\nts K}$ is the characteristic function of $K$ and
$\widecheck{.}$ denotes inverse Fourier transform. For all other $k$,
one has $A_{i} (k) = 0$.  As we shall see later, a more general
connection is possible via the FB coefficients of $\vL$ and
its subsets; see Eq.~\eqref{eq:FB} below for more.  Note that the
validity of \eqref{eq:ampli1} is a consequence of the uniform
distribution of $\vL^{\star}_{i}$ in $W^{}_{\! i} \ts $; compare the
detailed discussions in \cite[Sec.~7.1]{TAO} and
\cite{Martin1,Bob-ufd}. \exend
\end{remark}

Though Eq.~\eqref{eq:ampli1} looks nice, it is generally difficult to
calculate $\widecheck{1^{}_{W_i}}$ directly, due to the potentially
fractal nature of the window boundaries.  Let us thus turn to an
alternative approach of transfer matrix type that harvests the
inflation nature of our point sets.  In view of
Lemma~\ref{lem:disjoint}, Eq.~\eqref{eq:win-IFS} can now be 
rewritten as
\begin{equation}\label{eq:charwin}
     1^{}_{W_i} \, = \, \sum_{j=1}^{N} \sum_{t\in T_{ij}}
     1^{}_{Q \ts W_{\nts j} + t^{\star}} \ts ,
\end{equation}
to be understood in the Lebesgue sense (rather than pointwise).

\begin{theorem}\label{thm:cover}
  Let\/ $(W^{}_{1},\ldots,W^{}_{\! N})\in (\cK\RR^{d-1})^N$ be the
  unique solution to the contractive IFS~\eqref{eq:win-IFS}. Then,
  Eq.~\eqref{eq:charwin} holds in the Lebesgue sense for every\/
  $1\leqslant i\leqslant N$.

  Moreover, when considering the compact set\/
  $W\! =\bigcup_{i}W_{i}$, one has
\[
  \mf (y) \, \defeq  \sum_{i=1}^{N} 1^{}_{W_{i}}(y)
  \, = \, \mf (y)\, 1^{}_{W}(y) \ts ,
\]
where\/ $\ts\mf$ is a measurable, integer-valued function on\/
$\RR^{d-1}$ with\/ $\ts\supp (\mf) = W\nts\nts$. In particular, the
potential values on\/ $W$ are restricted to\/ $\{1,2,\ldots,N\}$.
\end{theorem}

\begin{proof}
  The validity of \eqref{eq:charwin} in the Lebesgue sense is clear
  from Lemma~\ref{lem:disjoint}.

  Since all $W_{\nts i}$ are compact, the function $\mf$ is well defined
  for all $y\in W$, and clearly $0$ outside of $W\!$. This means that
  we have $\mf \, 1^{}_{W} = \mf$ on all of $\RR^{d-1}$, while all
  remaining claims of the theorem are now immediate.
\end{proof}

\begin{remark}\label{rem:win-cover}
  Under some additional conditions, the function $\mf$ is constant for
  almost every\/ $y\in W\!$, with integer value\/ $\mc$.  In this case,
  one has $\sum_{i=1}^{N} \vol (W_i) = \mc \vol (W)$. This situation
  happens whenever the inflation defines a model set, where
  $\mc=1$. Beyond this case, one can have integer values of $\mc$,
  possibly up to $N\nts -1$.

  In general, however, the function $\mf$ need \emph{not} be constant
  almost everywhere in the total window $W\!$, as we shall see in the
  example of Eq.~\eqref{eq:winpic} in Section~\ref{sec:twisted}. This
  seems a significant difference to the covering degree of internal
  space by the dual inflation mentioned in Remark~\ref{rem:dual}. We
  shall return to this point in Section~\ref{sec:FB}. \exend
\end{remark}

Let us now switch to an analysis of the system \eqref{eq:charwin} of
equations after (inverse) Fourier transform, which turns it into a
rescaling equation for an $N$-tuple of continuous functions.

\section{Analysis of internal cocycle}\label{sec:cocycle}

Let us first recall a simple, but in our context vital, result on the
inverse Fourier transform of characteristic functions, which we prove
for convenience.

\begin{lemma}\label{lemma:affine}
  Let\/ $m\in\NN$ be fixed. Let\/ $K\subset \RR^m$ be compact and\/
  $Q\in \GL(m,\RR)$.  Then, one has the relation
\[
    \widecheck{1^{}_{Q K \nts + t}} (y) \, = \,
    \lvert\det (Q)\rvert \, \ee^{2 \pi \ii \langle t \ts | \ts y \rangle}
    \,\widecheck{1^{}_{\nts K}} (Q^T y) \ts ,
\]
which holds for all\/ $t, y \in\RR^m$, with continuity in both
variables.  
\end{lemma}

\begin{proof}
  With the change of variable $x = Q u + t$, one finds
\[
\begin{split}
  \widecheck{1^{}_{Q K \nts + t}} (y) \, & = 
  \int_{QK+t} \ee^{2 \pi \ii \langle x \ts |\ts y \rangle} \dd x
   = \, \lvert \det (Q)\rvert \, \ee^{2 \pi \ii \langle t \ts | \ts y \rangle}
  \! \int_{K} \ee^{2 \pi \ii \langle Qu \ts |\ts y \rangle} \dd u   \\[2mm]
  & = \, \lvert \det (Q)\rvert \, \ee^{2 \pi \ii \langle t \ts | \ts y \rangle}
  \! \int_{K} \ee^{2 \pi \ii \langle u \ts |\ts Q^T y \rangle} \dd u 
  \, = \, \lvert\det (Q)\rvert \,
     \ee^{2 \pi \ii \langle t \ts | \ts y \rangle}
      \,\widecheck{1^{}_{\nts K}} (Q^T y) \ts .
\end{split} 
\]
Continuity in $y$ follows from the Fourier transform of an
$L^{1}$-function being continuous, see \cite[Thm.~IX.7]{RS}, while
continuity in $t$ is obvious.
\end{proof}

Let $Q$ now be the linear map from \eqref{eq:def-Q} in internal 
space $\RR^{d-1}$ that is induced by the dilation 
$x \mapsto \lambda x$ in direct space, $\RR$. Set
$f^{}_{i}(y)\defeq \widecheck{1^{}_{W_{i}}}(y)$ and consider the
vector of functions
$|f(y)\rangle = | f^{}_{1} (y), \ldots , f^{}_{N} (y) \rangle $.
Also, let $\sB (y)$ be the \emph{internal Fourier matrix} that emerges
from the (inverse) Fourier transform of the $\star$-image of the
displacement matrix $T$, that is,
\[
    \sB^{}_{ij} (y) \, = \sum_{x\in
    T_{ij}} \ee^{2 \pi \ii \langle x^{\star} | \ts y \rangle} ,
\]
with $y\in\RR^{d-1}$. The matrix elements are again trigonometric
polynomials, this time generally multivariate, where one still has
\[
    \bigl| \sB_{ij} (y) \bigr| \, \leqslant \, M_{ij}
\]
for all $i,j$ and all $y\in\RR^{d-1}$,
in complete analogy to \eqref{eq:max-B}. With
Lemma~\ref{lemma:affine} and $\lvert \det (Q) \rvert = \lambda^{-1}$,
one now finds the following result of transfer matrix type via an
elementary computation; compare \cite[Sec.~3.2]{Pohl} for a
mathematically similar structure.

\begin{prop}\label{prop:FT}
Under inverse Fourier transform, Eq.~\eqref{eq:charwin} becomes
\[
     | f (y) \rangle \, = \, \lambda^{-1} \sB (y)\,
     | f (R \ts y) \rangle \ts ,
\]
with\/ $R=Q^T$ and\/ $\sB (y) = \widecheck{\delta^{}_{T^{\star}}} (y)$
as defined above, and all\/ $f^{}_{i}$ continuous.  \qed
\end{prop}

It is clear that $\lim_{y\to 0} \sB(y) = \sB (0) = M$. Moreover, from
the way it was constructed, we know that $R$ is a normal matrix and a
contraction. Consequently, its spectral norm agrees with its spectral
radius, see \cite[Sec.~2.3]{House}, and we have
$\theta \defeq \| R \ts\ts \|^{}_{2} = \rho (R) < 1$.  This leads to
the following property, where we use $\|.\|^{}_{2}$ also for the
$2$-norm of vectors.

\begin{lemma}\label{lem:contra}
  For\/ any\/ $\varepsilon >0$, there exists\/
  $\delta = \delta(\varepsilon) >0$ such that
\[
   \| \sB (R^{m} y) - M \|^{}_{2} \, < \, \theta^{m} \varepsilon
 \]
 holds simultaneously for all\/ $\| y \|^{}_{2} < \delta$ and all\/
 $m\in\NN$.
\end{lemma}

\begin{proof}
% old proof
% Since each element of $\sB (y)$ is analytic, one has the expansions
% \[
%  \sB^{}_{ij} (y) \, = \, M^{}_{ij} + \langle u^{}_{ij} |\ts y \ts \rangle
%  + \cO \bigl( \| y \|^{2}_{2} \bigr)
% \]
% near the origin, with $u^{}_{ij} = \nabla \sB^{}_{ij} (0)$, which are
% $N^2$ fixed vectors. Consequently, using
% $\| R^m y \|^{}_{2} \leqslant \| R^m \|^{}_{2} \ts \| y \|^{}_{2}
% \leqslant \| R \ts\ts \|^{m}_{2} \ts \| y \|^{}_{2} = \theta^m \| y
% \|^{}_{2}$ and the Cauchy--Schwarz inequality, we get
% \[
%  \big\lvert \sB^{}_{ij} (R^m y) - M^{}_{ij} \big\rvert 
%  \, \leqslant \, \theta^m \| u^{}_{ij} \|^{}_{2} \| y \|^{}_{2} +
%  \theta^{2m} \cO \bigl( \| y \|^{2}_{2} \bigr) ,
% \]
%from which the claim follows by standard arguments.
Recall that $\| A \|^{}_{2} \leqslant N
\sup_{i,j} \lvert A_{ij}\rvert$ holds for all $A\in \Mat (N,\CC)$.
Observe that
\[
\begin{split}  
  \bigl| \sB^{}_{ij} (R^m y) - M^{}_{ij} \bigr| \, & = \,
  \biggl| \sum_{x\in T_{ij}} \bigl( \ee^{2 \pi \ii \langle x^{\star}
    | \ts R^m y \rangle} - 1 \bigr) \biggr| \, \leqslant \,
  \sum_{x\in T_{ij}} \bigl|  \ee^{2 \pi \ii \langle x^{\star}
    | \ts R^m y \rangle} - 1 \bigr| \\[2mm]
  & = \sum_{x\in T_{ij}} 2 \ts\ts \bigl| \ts \sin \bigl(\pi
  \langle x^{\star} | \ts R^m y \rangle \bigr) \bigr| \,
  \leqslant \, 2 \pi \! \sum_{x\in T_{ij}}  \bigl|\langle x^{\star} |
  \ts R^m y \rangle \bigr| ,
\end{split}
\]
where we have used some trigonometric identities and the fact that
$\lvert \ts \sin (z) \rvert \leqslant \lvert z \rvert$ holds for all
$z\in\RR$. Combining this with the Cauchy--Schwarz inequality
\[
  \bigl|\langle x^{\star} | \ts R^m y \rangle \bigr| \, \leqslant \,
  \| x^{\star} \|^{}_{2} \ts \| R^m y \|^{}_{2} \, \leqslant \,
  \| x^{\star} \|^{}_{2} \, \theta^m \| y \|^{}_{2}
\]
gives the claim by standard arguments.
\end{proof}

Now, for $n\in\NN$, we can define an \emph{internal cocycle} 
from the Fourier matrix via
\[
   \sB^{(n)}(y) \, \defeq \,
   \sB(y) \sB(R \ts y) \cdots \sB(R^{n-1} y) \ts ,
\]
with $\sB^{(1)} = \sB$ and $\sB^{(n)}(0)=M^n$. In analogy to
\eqref{eq:cocycle1}, we now have
\begin{equation}\label{eq:cocycle2}
    \sB^{(n+1)} (y) \, =\,  \sB^{(n)} (y) \, \sB( R^n y) 
  \, = \, \sB(y) \, \sB^{(n)} (Ry) 
\end{equation}
for all $n \geqslant 1$, and also
$\sB^{(n+m)} (y) = \sB^{(n)} (y) \sB^{(m)} (R^n y)$ for all
$m \geqslant 1$ and $n\geqslant 0$, the latter with the convention
$\sB^{(0)} \defeq \one$, which we adopt from now on. Clearly, one has
$\lvert \sB^{(n)}_{ij} (y) \rvert \leqslant (M^{n})^{}_{ij}$, and
Fact~\ref{fact:B-mat} remains valid with $B^{(n)}$ replaced by the
internal cocycle $\sB^{(n)}$.

Next, we want to consider the matrix function defined by
\begin{equation}\label{eq:C-def}
    C(y) \, \defeq \lim_{n\to\infty} \beta^{n} \sB^{(n)}(y) \ts ,
\end{equation}
with $\beta = \lvert \det (R) \rvert$, where $\beta\ts \lambda = 1$
because $\lambda$ is a unit.  We thus need to establish that $C (y)$
is well defined as a limit, for every $y\in\RR^{d-1}$. To this end, we
employ the $2$-norm for vectors and the corresponding operator norm,
both denoted by $\|.\|^{}_{2}$ as before.

\begin{prop}\label{prop:equi}
  The sequence\/
  $\bigl(\beta^n ( \sB^{(n)} (y) - M^n ) \bigr)_{n\in\NN}$ of matrix
  functions is equicontinuous at\/ $y=0$, which is to say that
\[
    \forall \varepsilon > 0 \! : \, \exists \ts \delta = \delta 
    (\varepsilon)>0 \! : \, \forall n\in\NN \! : \,
    \bigl( \| y \|^{}_{2} < \delta \, \Longrightarrow \, \beta^n 
    \| \sB^{(n)} (y) - M^n \|^{}_{2} < \varepsilon \bigr) .
\]     
\end{prop}

\begin{proof}
  Since $M$ is non-negative, the spectral radius of $M^n$ is
  $\rho(M^n) = \lambda^n$, for all $n\in\NN_0$. Let us first consider
  the case that $M$ is normal. Then, we can most easily work with the
  spectral norm, because we have
  $\| M^n \|^{}_{2} = \rho(M^n) = \lambda^n$ for $n\geqslant 0$, hence
  $\| M^n\|^{}_{2} = \| M \|^{n}_{2}$. Also, we get
\[
   \| \sB(y) \|^{}_{2}  \,\leqslant\, 
   \big\| \ts\ts \lvert\sB(y)\rvert \ts\ts \big\|_{2}
   \,\leqslant\, \| M \|^{}_{2}
\]
for all $y$ in this case. This estimate holds because the spectral
norm has the required monotonicity property; compare \cite[Thm.~1]{JN}
or \cite[Exc.~5.6.P42]{HJ}.

Now, harvesting the cocycle property \eqref{eq:cocycle2}, a simple
telescopic argument leads to
\[
   \sB^{(n)}(y)-M^{n} \, = \sum_{\ell=0}^{n-1} M^{\ell} \ts
   \bigl( \sB(R^{\ell}y) - M \bigr) \,\sB^{(n-1-\ell)}(R^{\ell+1}y)
\]
for $n\geqslant 1$, with $\sB^{(0)} = \one$ as above. Via the triangle
inequality, using the above properties, one then finds the estimate
\[
  \| \sB^{(n)} (y) - M^n \|^{}_{2} \, \leqslant \,
  \| M \|^{n-1}_{2} \sum_{\ell=0}^{n-1} \| \sB (R^{\ell} y) - M \|^{}_{2}
  \, = \, \lambda^{n-1} \sum_{\ell=0}^{n-1}
  \| \sB (R^{\ell} y) - M \|^{}_{2} \ts .
\]
For $\varepsilon > 0$ and $\| y \|^{}_{2} < \delta$, with
$\beta = \lambda^{-1}$ and the $\delta$ from Lemma~\ref{lem:contra},
this gives
\[
  \big\| \beta^n \bigl( \sB^{(n)} (y) - M^n \bigr)
  \big\|_{2} \, \leqslant \,
  \beta \sum_{\ell=0}^{n-1} \theta^{\ell} \varepsilon \, \leqslant \,
  \frac{\beta \ts \varepsilon}{1 - \theta}
\]
by a geometric series argument, which establishes the claim when $M$
is normal.

For the general case, we employ a different sub-multiplicative matrix
norm, which depends on $M$ and again satisfies $\| M \| = \rho (M)$,
hence $\| M^n \| \leqslant \| M\|^n = \lambda^n$. Following
\cite[Sec.~2.4]{House}, one such norm can simply be constructed as
follows. Consider the convex body
\[
  K \, = \, \diag (v^{}_{1}, \ldots , v^{}_{N})
  \{ x \in \CC^N : \| x \|^{}_{\infty} \leqslant 1 \} \ts ,
\]
where the $v^{}_i$ are the strictly positive entries of
$|\ts v\rangle$, the (frequency normalised) right PF eigenvector of
$M$, and define
\[
  \| x \|^{}_{v} \, \defeq \, \inf
  \big\{ \alpha>0 : x \in \alpha K \big\}
  \, =  \max_{ 1\leqslant i \leqslant N} 
  \frac{\lvert x_i \rvert}{v_i} \ts .
\]
This is a matrix norm on $\CC^N$ that is \emph{absolute}, so
$\| x \|^{}_{v} = \big\|\ts \lvert x \rvert \ts \big\|_{v}$ for all
$x\in\CC^N$, with $\lvert x \rvert$ denoting the vector with entries
$\lvert x_i \rvert$. Now, let $\| . \|^{}_{K}$ denote the matching
operator norm on $\Mat(N,\CC)$, as defined by
\[
   \| A \|^{}_{K} \, \defeq  \sup_{\| x \|^{}_{v} =1} \| Ax\|^{}_{v} \ts ,
\]
which is sub-multiplicative and satisfies
$\| M \|^{}_{K} = \rho (M) = \lambda$ by construction.  What is more,
it also satisfies the monotonicity property
$\| A \|^{}_{K} \leqslant \big\| \ts \lvert A \rvert \ts \big\|_{K}$
for all $A\in \Mat (N,\CC)$, again by \cite[Thm.~1]{HJ}.
Consequently, we still get
$\| \sB^{(n)} (y) \|^{}_{K} \leqslant \| M^n \|^{}_{K} \leqslant \| M
\|^{n}_{K}$ for all $y\in\CC^{d-1}$ and all $n\in\NN$.

Equipped with this matrix norm, we can repeat our previous telescopic
argument, now leading to the estimate
\[
  \big\| \beta^n \bigl( \sB^{(n)} (y) - M^n \bigr) \big\|_{K}
  \, \leqslant \, \beta \sum_{\ell=0}^{n-1}
  \big\| \sB (R^{\ell}y) - M \big\|_{K} \ts . 
\]
From here, since the vector norms $\|.\|^{}_{2}$ and $\|.\|^{}_{v}$
are equivalent, as are the matrix norms $\| . \|^{}_{2}$ and
$\|.\|^{}_{K}$, we can adjust the choice of
$\delta = \delta (\varepsilon)$ to reach the same conclusion.
\end{proof}

Combining the equicontinuity of $\beta^n \sB^{(n)}(y)$ at $0$ from
Proposition~\ref{prop:equi} with $\sB^{(n)} (0) = M^n$ and
Fact~\ref{fact:projector}, a standard
$2\ts \varepsilon$-argument gives the following consequence.

\begin{coro}\label{coro:projector}
  Let\/ $P$ be the projector from\/ \eqref{eq:def-P} and\/
  $\sB^{(n)} (y)$ the internal cocycle. Then, for all\/
  $\varepsilon > 0$, there exists\/
  $\delta^{\ts\prime}= \delta^{\ts\prime} (\varepsilon)>0$ and\/
  $n^{}_{0} = n^{}_{0} (\varepsilon)$ such that
\[
   \| \beta^n \sB^{(n)} (y) - P \|^{}_{2} \, < \: \varepsilon
\]
holds for all integer\/ $n \geqslant n^{}_{0}$ and all\/ $y\in\RR^{d-1}$ 
with\/ $\| y \|^{}_{2} < \delta^{\ts\prime}$.  \qed
\end{coro}

Now, we are set to establish the convergence of our internal cocycle
as follows.

\begin{theorem}\label{thm:B-conv}
  The scaled internal cocycle sequence\/
  $\bigl(\beta^n \sB^{(n)} (y) \bigr)_{n\in\NN}$ converges compactly
  on\/ $\RR^{d-1}$. Consequently, the matrix function\/ $C (y)$ from\/
  \eqref{eq:C-def} is well defined and continuous.
\end{theorem}

\begin{proof}
  Let $K\subset \RR^{d-1}$ be compact, choose $\varepsilon > 0$, and
  let $\delta = \delta (\varepsilon) > 0$ be as in
  Proposition~\ref{prop:equi}.  We will establish the claim by showing
  that the sequence is uniformly Cauchy on $K$.

  For $p,q,r \in \NN$, we employ the cocycle property from
  \eqref{eq:cocycle2} to get
\begin{equation}\label{eq:start}
\begin{split}
   \| \beta^{\ts p+q} \sB^{(p+q)}\nts (y)  & - 
   \beta^{\ts p+q+r} \sB^{(p+q+r)} \nts (y) \|^{}_{2} \\[2mm]
   & \leqslant \, \| \beta^{\ts p} \sB^{(p)}\nts (y) \|^{}_{2}
    \, \| \beta^{q} \sB^{(q)} \nts (R^{\ts p} y) -
   \beta^{q+r} \sB^{(q+r)} \nts (R^{\ts p} y) \|^{}_{2} \ts ,
\end{split}
\end{equation}
where the first factor on the right is bounded by
$\beta^{\ts p} \| M^p \|^{}_{2}$ and thus uniformly bounded by a
constant $c^{}_{\nts B}$, as a consequence of Fact~\ref{fact:B-mat},
applied to $\sB^{(n)}$ with the spectral norm. Via the triangle inequality, 
the second factor on the right-hand side of \eqref{eq:start} is 
bounded by
\begin{equation}\label{eq:step-one}
  \big\| \beta^{q} \bigl( \sB^{(q)} (R^{\ts p} y)
         - M^{q}\bigr) \big\|^{}_{2} + 
  \big\| \beta^{q} M^{q} - \beta^{q+r} M^{q+r} \big\|^{}_{2} +
  \big\| \beta^{q+r} \bigl(\sB^{(q+r)} ( R^{\ts p} y )
     - M^{q+r} \bigr) \big\|^{}_{2} \ts .
\end{equation}
Choose $p$ large enough so that $R^{\ts p} K$ is contained in the open
ball of radius $\delta$ around $0$, which is possible because $R$ is a
contraction. Then, the first term in \eqref{eq:step-one}, as well as
the last, is bounded by $\varepsilon$. Since
$(\beta^n M^n)^{}_{n\in\NN}$ converges to $P$ by
Fact~\ref{fact:projector}, where $\beta = \lambda^{-1}$, the sequence
is Cauchy, so there is a $q^{}_{0}\in\NN$ such that the middle term in
\eqref{eq:step-one} is bounded by $\varepsilon$, for all
$q\geqslant q^{}_{0}$ and $r \in \NN$.  Consequently,
\eqref{eq:step-one} is bounded by $3 \ts \varepsilon$ for the chosen
$p$, all $q\geqslant q^{}_{0}$, and all $r\in\NN$. Via
Fact~\ref{fact:B-mat}, used with $\sB^{(n)}$ instead of $B^{(n)}$,
this gives an upper bound of
$3 \ts\ts c^{}_{B} \ts\ts \varepsilon$ to the left-hand side of
\eqref{eq:start}. As this bound is independent of $y\in K$, and
$\varepsilon>0$ was arbitrary, uniform convergence on $K$ follows.

Since we have a compactly convergent sequence of matrix functions,
each of which is analytic and thus certainly continuous, the last
claim is obvious.
\end{proof}

Let us next analyse the matrix function $C (y)$, where we know
\[
  C(0) \, = \, P 
\]
from Fact~\ref{fact:projector}. Now, Eq.~\eqref{eq:cocycle2} implies
that, for any fixed $m\in\NN$,
\[
  C (y) \, = \lim_{n\to\infty} \beta^{n+m} \sB^{(n+m)} (y)
  \, = \, C(y) \, \beta^{m} \lim_{n\to\infty} \sB^{(m)} (R^{n} y)
  \, = \, C(y)\, \beta^{m} M^m,
\]
because $R$ is a contraction and $\sB^{(m)} (0) = M^m$.
With $m=1$, this gives
\[
   C (y) M \, = \, \lambda \ts\ts C (y) \ts ,
\]
as well as $C(y) = C(y) P$ from taking the limit $m\to\infty$.  Each
row of $C (y)$ thus is a left eigenvector of $M$ for its eigenvalue
$\lambda$, or vanishes, hence is a $y$-dependent multiple 
of $\langle u \ts |$. But this means
\begin{equation}\label{eq:C-structure}
    C(y) \, = \, |\ts c(y)\rangle\langle u\ts |
\end{equation}
with $|\ts c(0)\rangle = |\ts v\rangle$. Observing that the window
volumes are proportional to the entries of $| \ts v \rangle$ by
Lemma~\ref{lem:disjoint}, so $|f(0)\rangle = \eta \ts | \ts v\rangle$
for some $\eta > 0$, one has the following consequence.

\begin{coro}\label{coro:FT-win}
  For any\/ $y\in\RR^{d-1}$, the matrix\/ $C(y)$ from \eqref{eq:C-def}
  has rank\/ $\leqslant 1$, and can be represented as in
  \eqref{eq:C-structure}. Moreover, with\/
  $f^{}_i = \widecheck{1^{}_{W_i}}$, one has\/
  $| f (y) \rangle = \eta \ts\ts | \ts c(y) \rangle$ with the above\/
  $\eta$, together with\/
  $|\ts c(y) \rangle = C(y) \ts |\ts v \rangle$.  \qed
\end{coro}

In particular, this result makes the functions $f^{}_i$ effectively
computable from $C$.

\begin{remark}
  Let us mention that the continuity of the functions $f^{}_{i}$ is
  also clear from the fact that each is the (inverse) Fourier
  transform of an $L^1$-function, and $f^{}_{i}$ decays at infinity by
  the Riemann--Lebesgue lemma; see \cite[Thm.~IX.7]{RS}.  What is
  more, since all $W_i$ are compact, we actually know that each
  $f^{}_{i} $ has an analytic continuation to an entire analytic
  function of $d\nts -\! 1$ variables, with a well-known growth
  estimate according to the Paley--Wiener theorem; see
  \cite[Thm.~IX.12]{RS}. This also means that
  the rank of $C(y)$ is $1$ almost everywhere.  \exend
\end{remark}

\section{Fourier--Bohr coefficients and uniform
distribution}\label{sec:FB}

Here, we explain the general connection with the diffraction
amplitudes mentioned earlier in Remark~\ref{rem:model}.  Given a typed
point set $\vL = \dot{\bigcup}_{i} \, \vL_{i} \subset \RR$, its
\emph{Fourier--Bohr (FB) coefficient} (or amplitude) at $k\in\RR$ is
defined as a volume-averaged exponential sum,
\begin{equation}\label{eq:FB}
   A^{}_{\nts\vL}(k) \, \defeq \lim_{r\to\infty} \myfrac{1}{2r}
   \sum_{\substack{x\in\vL \\ \lvert x\rvert\leqslant r}} \ee^{-2\pi\ii kx}, 
\end{equation}
and similarly for the control point sets $\vL_{i}$ with
$1\leqslant i \leqslant N$, provided the limits exist. This is the
case for point sets from primitive inflation rules, which are linearly
repetitive and thus uniquely ergodic \cite{Boris,LP}. The definition
entails that $A^{}_{\nts\vL} (0) = \dens (\vL)$, and one gets
$\sum_{i=1}^{N} A^{}_{\nts \vL_{i}} (k) = A^{}_{\nts \vL} (k)$ for all
$k\in\RR$ because the point sets $\vL_{i}$ are disjoint by
construction. Let us also recall that $A^{}_{\nts\vL} (.)$, when
viewed as a function of $\vL$, is continuous, which correponds to the
continuity of all eigenfunctions in this setting \cite{Daniel}; see
Remark~\ref{rem:eigen} below for more.

In general, we know from the embedding procedure that
$\overline{\vL^{\star}_{i}} =W^{}_{\nts i}$. If $\vL_i$ is also a
model set, the point set $\vL^{\star}_{i}$ is uniformly distributed
(and even well distributed) in $W_i$; compare \cite{Martin1,Bob-ufd}. 
This uniform distribution occurs more generally, as we analyse next.

It is clear from Remark~\ref{rem:win-cover} and the example in
Section~\ref{sec:twisted} that the lift of $\vL$ to internal space
will not be uniformly distributed in $W$ in general. However, the
situation is more favourable for the individual point sets $\vL_i$.
For any $1 \leqslant i \leqslant N$, consider the sequence
$(\mu^{(n)}_{i})^{}_{n\in\NN}$ of point measures in internal space
defined by
\[
  \mu^{(n)}_{i} \, = \, \myfrac{1}{2n}
  \sum_{\substack{x\in\vL_i \\ \lvert x \rvert \leqslant n}}
  \delta^{}_{x^{\star}} \ts .
\]
Clearly, one has $\supp (\mu^{(n)}_{i}) \subset W_{i}$ by
construction, and $\mu^{}_{i} \defeq \lim_{n\to\infty} \mu^{(n)}_{i}$
exists (under weak convergence), due to the strict ergodicity of the
dynamical system defined by $\vL_i$. An explicit argument for this
convergence, based on the linear repetitivity of $\vL_i$, can be
formulated along the lines of the proof of \cite[Thm.~5.1]{LP},
observing that $\bigl( \delta^{}_{x^{\star}} \nts * g \bigr) (y) =
g (y - x^{\star})$ and using the result pointwise. Here,
$\mu^{}_{i}$ is a positive measure on $\RR^{d-1}$ with
$\supp (\mu^{}_{i}) \subseteq W^{}_{i}$ and total mass
$\| \mu^{}_{i} \| = \dens (\vL_{i})$. We say that $\vL_i$
\emph{induces} the measure $\mu^{}_{i}$ in internal space.

\begin{remark}
   If $g \in C^{}_{0} (\RR^{d-1})$ and $a,b \in \RR$ with $a<b$,
   one can consider, for each fixed $i$,
\[
    w_i \bigl( [a,b] \bigr) \, \defeq \! \sum_{x\in \vL_i \cap [a,b]}
    g (x^{\star}) \ts .
\]   
  Now, since any $g \in C^{}_{0} (\RR^{d-1})$ is bounded, 
  the Delone property of $\vL_i$ implies that there are some 
  numbers $c,d>0$ such that $b-a>c$ implies
\[
    \bigl| w_i \bigl( [a,b]\bigr) \bigr| \, \leqslant \,
    (b-a) \ts d \ts ,
\]  
  and each $w_i$ is a local weight function in the sense of
  \cite{LP}. 
  
  Then, \cite[Thm.~5.1]{LP} yields convergence of
  $w_i \bigl( t+ [a,b]\bigr)/(b-a)$, uniformly in $t\in \RR$, due to
  the linear repetitivity of the $\vL_i$. An analogous argument
  applies when $g$ is replaced by $\delta^{}_{z} * g$ with
  an arbitrary $z\in\RR^{d-1}$. 
  The result obtained this way is stronger than needed below, 
  and actually also gives the uniform existence of the FB coefficients.
\exend
\end{remark}

When $\vL_i$ induces $\mu^{}_{i}$, a simple calculation (with a change
of the summation variable) shows that $\lambda \vL_i + t$ with
$t\in \ZZ[\lambda]$ induces the positive measure
\[
  \myfrac{1}{\lambda} \, \delta^{}_{t^{\star}}
  \nts * (Q . \mu^{}_{i}),
\]
where $Q$ is the contraction from \eqref{eq:def-Q} and $Q.\mu$
denotes the push-forward of a finite measure $\mu$, so
$\bigl( Q. \mu \bigr) (\varphi) = \mu (\varphi \circ Q)$ for
$\varphi \in C^{}_{0} (\RR^{d-1})$.  Equivalently, one can use
$\bigl( Q. \mu \bigr) (\cE) = \mu \bigl( Q^{-1} (\cE)\bigr)$ with
$\cE$ an arbitrary Borel set.

Let us now assume that our typed point set
$\vL = \dot{\bigcup}_{i} \vL_{i}$ is a fixed point of the inflation
equation \eqref{eq:set-inflation}. This is no restriction as one can
always achieve this via replacing $\varrho$ by a suitable power;
compare Remark~\ref{rem:cycle}. Then, our induced measures
$\mu^{}_{1}, \ldots , \mu^{}_{N}$ must satisfy
\begin{equation}\label{eq:mu-eq}
  \mu^{}_{i} \, = \, \myfrac{1}{\lambda} \sum_{j=1}^{N}
  \, \sum_{t\in T_{ij}} \delta^{}_{t^{\star}} \nts *
  (Q. \mu^{}_{j}) \ts ,
\end{equation}
which defines a system of $N$ linear equations. We can spell out one
solution as follows, where $\mu^{}_{\mathrm{Leb}}$ denotes Lebesgue
measure on internal space, $\RR^{d-1}$.

\begin{lemma}\label{lem:sol}
  The absolutely continuous measures\/ $\mu^{\prime}_{i} = g^{}_{i}
  \ts \mu^{}_{\mathrm{Leb}}$ with Radon--Nikodym densities
\[
  g^{}_{i} \, = \, \frac{\dens (\vL_i)}{\vol (W_i)}
  \, 1^{}_{W_i}
\]
satisfy\/ \eqref{eq:mu-eq} together with\/ $\| \mu^{\prime}_{i} \|
= \dens (\vL_i)$.  
\end{lemma}

\begin{proof}
  Observe that
  $Q.(1^{}_{W_i} \ts \mu^{}_{\mathrm{Leb}}) = \lvert \det (Q)
  \rvert^{-1} 1^{}_{Q W_i} \, \mu^{}_{\mathrm{Leb}}$, which follows
  from a simple change of variable calculation.  Likewise, one has
  $\delta^{}_{t^{\star}}\nts * 1^{}_{W_i} = 1^{}_{W_i + t^{\star}}$,
  and inserting the expressions into \eqref{eq:mu-eq} leads to
\[
  \frac{\dens (\vL_i)}{\vol (W_i)} \,  1^{}_{W_i} \, =
  \sum_{j=1}^{N} \, \sum_{t\in T_{ij}} \frac{\dens (\vL_j)}{\vol (W_j)}
  \, 1^{}_{Q W_j + t^{\star}} \ts .
\]
By construction, we have $\dens (\vL_i) = \dens (\vL) \ts v^{}_{i}$,
where $v^{}_{i}$ is the relative frequency of points of type $i$;
compare Remark~\ref{rem:geo}. On the other hand, we know from
Lemma~\ref{lem:disjoint} that the $N$ window volumes satisfy
$\vol (W_i) = \eta \ts v^{}_{i}$ for some fixed $\eta > 0$, which
implies that
\[
    \frac{\dens (\vL_i)}{\vol (W_i)} \, = \, 
    \frac{\dens (\vL)}{\eta}
\]
is independent of $i$, and the previous equation turns into the window
equation \eqref{eq:charwin}, which is satisfied in the Lebesgue sense.

  The claimed normalisation is obvious.
\end{proof}

Now, we interpret the right-hand side of \eqref{eq:mu-eq} as a linear
mapping on $\bigl(\cM_{+} (\RR^{d-1})\bigr)^{N}$, with
$\cM_{+} (\RR^{d-1})$ denoting the finite, positive measures on
$\RR^{d-1}$, equipped with the total variation norm, $\| . \|$.  If
$(\mu^{}_{1}, \ldots , \mu^{}_{N})$ is an $N$-tuple of positive
measures, its image is $(\mu^{\prime}_{1}, \ldots , \mu^{\prime}_{N})$
with
\[
\begin{split}
  \| \mu^{\prime}_{i} \| \, & = \, \Big\| \myfrac{1}{\lambda}
  \sum_{j=1}^{N} \, \sum_{t\in T_{ij}} \delta^{}_{t^{\star}} \nts *
  (Q.\mu^{}_{j}) \Big\| \, = \, \myfrac{1}{\lambda} \sum_{j=1}^{N}
  \, \sum_{t\in T_{ij}} \big\| \delta^{}_{t^{\star}}\nts * (Q.\mu^{}_{j})
  \big\| \\[2mm]
  & = \, \myfrac{1}{\lambda} \sum_{j=1}^{N} M^{}_{ij}
  \big\| Q. \mu^{}_{j} \big\| \, = \, \myfrac{1}{\lambda}
  \sum_{j=1}^{N} M^{}_{ij} \ts \| \mu^{}_{j} \| \ts .
\end{split}
\]
Consequently, when $\| \mu^{}_{i} \| = \alpha \ts v^{}_{i}$ for all
$1 \leqslant i \leqslant N$ and some $\alpha > 0$, the total mass of
each $\mu^{}_{i}$ is preserved under the iteration because
$M | \ts v \rangle = \lambda \ts | \ts v \rangle$. This leads to the
following result.

\begin{prop}\label{prop:m-space}
  Let\/ $\alpha>0$ be fixed and consider the space
\[
  \cM_{\alpha} \, \defeq \, \big\{
  (\nu^{}_{1}, \ldots , \nu^{}_{N}) : \nu^{}_{i} \in \cM_{+}
  (\RR^{d-1}) , \, \| \nu^{}_{i} \| = \alpha \ts v^{}_{i} \big\} ,
\]
with\/ $|\ts v \rangle$ the right PF eigenvector of\/ $M$.  Then,
$\cM_{\alpha}$ is invariant under the iteration of the right-hand side
of\/ \eqref{eq:mu-eq}, and contains precisely one solution to
Eq.~\eqref{eq:mu-eq}, namely the one defined by\/
$\nu^{}_{i} = \frac{\alpha \, v^{}_{i}}{\vol (W_i)} \ts 1^{}_{W_i} \ts
\mu^{}_{\mathrm{Leb}}$ for\/ $1 \leqslant i \leqslant N$.
\end{prop}

\begin{proof}
  The space $\cM_{\alpha}$ can be equipped with the Hutchinson metric,
  compare \cite[Sec.~2]{BM00} and references therein, which turns it
  into a complete metric space. The iteration then is a contraction,
  as is obvious from 
\[
  \Big\| \myfrac{1}{\lambda} \, \delta^{}_{t^{\star}}\nts *
  (Q.\ts\nu^{}_{j} ) \Big\| \, = \, \myfrac{1}{\lambda} \,
  \big\| \delta^{}_{t^{\star}} \nts * (Q.\ts\nu^{}_{j}) \big\|
  \, = \, \myfrac{1}{\lambda} \, \| Q . \ts\nu^{}_{j} \|
  \, = \, \myfrac{1}{\lambda} \, \| \nu^{}_{j} \|
\]
where $\lambda > 1$; see \cite[Sec.~5]{BM00} for the remaining steps.

Now, the first claim is a consequence of Banach's contraction
principle, while the concrete form of the solution follows from
Lemma~\ref{lem:sol}.
\end{proof}

If one starts the iteration with an arbitrary $N$-tuple of
non-negative measures, not all $0$, there is a unique component of the
total mass vector in the PF direction of $M$, which defines the
parameter $\alpha$, and all other components decay exponentially fast.

Our main result of this section can now be formulated as follows.

\begin{theorem}\label{thm:uniform-dist}
  Let\/ $\vL = \dot{\bigcup}_{i} \vL_{i}$ be the typed point set of a
  primitive, unimodular PV inflation rule as constructed above, and
  consider the natural CPS that emerges from the Minkowski embedding.
  Then, each\/ $\vL_i$ induces a unique measure in internal space,
  namely
\[
  \mu^{}_{i} \, = \,   \frac{\dens (\vL_i)}{\vol (W_i)}
      \, 1^{}_{W_i} \ts \mu^{}_{\mathrm{Leb}} \ts ,
\]
where the\/ $W_i$ are the solutions of the window IFS\/
\eqref{eq:win-IFS}.  This entails the statement that, for all\/
$1\leqslant i \leqslant N$, the set\/ $\vL^{\star}_{i}$ is uniformly
distributed in\/ $W_i$.
\end{theorem}

\begin{proof}
  The IFS \eqref{eq:mu-eq} for the distributions induced by the
  $\vL_i$ on the compact sets $W_i$ is contractive on $\cM_{\alpha}$,
  with $\alpha = \dens (\vL)$, and the unique solution is the one
  stated.

  Recalling the definition of the induced measures, weak convergence
  clearly is equivalent to the uniform distribution of $\vL^{\star}_{i}$
  in $W_i$.
\end{proof}

At this point, we can return to the connection between the FB
coefficients and the Fourier transform of the windows, even though the
latter generally only code a covering model set. Still, due to uniform
distribution, one obtains an explicit formula as follows.

\begin{coro}
  Under the assumptions of 
  Theorem~\textnormal{\ref{thm:uniform-dist}}, the FB coefficients
  of the\/ $\vL_i$ are proportional to the Fourier amplitudes of the
  covering model set via
\[
    A^{}_{\nts\vL_i}(k) \, = \, 
    \myfrac{\dens (\vL_i)}{\vol (W_i)}\,
    \widecheck{1^{}_{W_{i}}}(k^{\star}) 
\]
for any\/ $k\in L^{\circledast}$, together with 
$A^{}_{\nts\vL_{i}}(k)=0$ for any 
$k\in\RR\setminus L^{\circledast}$. \qed
\end{coro}

In the special situation that the function
$\mf$ from Remark~\ref{rem:win-cover} satisfies $\mf (y) = \mc$ for
a.e.\ $y\in W\!$, one further gets
\begin{equation}\label{eq:genampli}
    A^{}_{\nts\vL_i}(k) \, = \, 
    \frac{\dens (\cL)}{\mc}\,
    \widecheck{1^{}_{W_{i}}}(k^{\star}) \ts ,
\end{equation}
which reduces to the standard formula for model sets when
$\mc = 1$.

\begin{remark}\label{rem:eigen}
  If we interpret the FB coefficient $A^{}_{\nts\vL}(k)$ as
  a function of $\vL$, one obtains
\[ 
   A^{}_{\ts t\ts +\vL}(k) \, = \, \ee^{-2\pi\ii k t}\, A^{}_{\nts\vL}(k)\ts .
\]
Consequently, whenever the coefficient does not vanish, this defines
an eigenfunction of the strictly ergodic dynamical system $(\YY,\RR)$,
where $\YY$ is the hull of $\vL$ obtained as the closure of the
translation orbit $\{ t + \vL : t \in \RR \}$ in the local topology;
compare \cite[Ch.~4]{TAO}. The analogous connection exists with the
$A_{\vL_i}$ for $1 \leqslant i \leqslant N$, not all of which can
vanish simultaneously for any given $k\in L^{\circledast}$. This
explains why $L^{\circledast}$ is the pure point part of the dynamical
spectrum (in additive notation) and how the diffraction intensities
are connected with the eigenfunctions; see \cite{BL-review,Daniel} and
references therein for more.

Both for regular model sets and for primitive inflation tilings, it is
known that the eigenfunctions on $\YY$ have continuous
representatives; see \cite{Daniel} and references therein. This also
means that the dynamical point spectrum for such systems is the same
in the topological and in the measure-theoretic sense.  \exend
\end{remark}

Whenever constant covering of the total window is satisfied in our
setting, we have the following consequence for the FB coefficients,
where we use $c (y)$ from \eqref{eq:C-structure} and
Corollary~\ref{coro:FT-win}.

\begin{coro}
  Assume that the total window covering is almost surely constant.
  Then, the FB coefficients, for\/ $k\in L^{\circledast}$, are
  obtained as
\[
  A^{}_{\nts \vL_{i}} (k) \, = \, \dens (\vL) \, c^{}_{i} (k^{\star}) \ts ,
\]
and vanish for all other\/ $k$.\qed
\end{coro}

The corresponding diffraction intensities follow from
Eq.~\eqref{eq:intens}.  Note that the covering degree does \emph{not}
show up in this relation.  The intensity at any wave number
$k\in L^{\circledast}$ can efficiently be approximated by truncating
the infinite product representation for $C(k^{\star})$ and calculating
the amplitudes as explained above.

At this point, we turn to some applications of the cocycle method to
concrete inflation systems on the real line, which will illustrate the
above results.

\section{Examples -- the Pisa substitutions}\label{sec:examples}

Let us introduce an interesting family of primitive inflations as
follows, based on the alphabet
$\cA = \{ a^{}_{1}, \ldots , a^{}_{d} \}$ with $d\geqslant 2$. The
explicit rule is given by $a^{}_{i} \mapsto a^{}_{1} a ^{}_{i+1}$ for
$1\leqslant i \leqslant d-1$, together with
$a^{}_{d} \mapsto a^{}_{1}$. In short, we have
$\varrho^{}_{d} = (a^{}_{1} a^{}_{2}, a^{}_{1} a^{}_{3} , \ldots ,
a^{}_{1} a^{}_{d}, a^{}_{1} )$. We call
$\{ \varrho^{}_{d} : d \geqslant 2 \}$ the family of \emph{Pisa
  substitutions}.  For $d=2$, this is the classic Fibonacci rule,
while $d=3$ is known as the Tribonacci substitution in the literature;
see \cite{PFBook} and references therein.

Let us first collect some general results for this family. The
substitution matrix reads
\[
  M_{d} \, = \, \begin{pmatrix} 1 & 1 & 1 & \dots & 1 & 1 \\
    1 & 0 & 0 & \dots & 0 & 0 \\ 0 & 1 & 0 & \dots & 0 & 0 \\
    0 & 0 & 1 & \dots & 0 & 0 \\
    \vdots &  \vdots & & \ddots &  & \vdots \\
    0 & 0 & 0 & \dots & 1 & 0
 \end{pmatrix}
\]
with $\det (M_{d}) = (-1)^{d-1}$. Note that $M_d$ is not normal for
$d\geqslant 3$, whence we need Proposition~\ref{prop:equi} in the
generality stated and proved. The characteristic polynomial of $M_{d}$
is
\[
  p^{}_{d}(x) \, =\, x^d - (1 + x + x^2 + \dots + x^{d-1})\ts .
\]
By \cite[Thm.~2]{Brauer}, $p^{}_{d}$ is irreducible, with one root
$>1$, which is the PF eigenvalue $\lambda^{}_{d}$ of $M_{d}$, and all
others inside the unit disk. So, $\lambda^{}_{d}$ is a PV unit of
degree $d$, which satisfies $\lim_{d\to\infty}\lambda^{}_{d}=2$. The
discriminant of $p^{}_{d}$ for $d\geqslant 2$ is given by
\[
   \Delta^{}_{d} \, = \, 
   (-1)^{\frac{d(d+1)}{2}}\, \frac{(d+1)^{d+1} -2 \ts (2d)^d}{(d-1)^2} \ts ,
\]
which is due to M.~Alekseyev; see \cite[A{\ts}106273]{OEIS} for details.

The right PF eigenvector is denoted by $|\ts v \rangle$ as before,
where we now drop the dependence on $d$ for ease of notation.  When
normalised as $\langle 1 | \ts v \rangle = 1$, it reads
\[
  \lvert\ts v\rangle \, = \, 
  \bigl(\lambda^{-1},\lambda^{-2},\lambda^{-3},\dots,
   \lambda^{-d+1},\lambda^{-d}\bigr)^{T}.
\]
The corresponding left PF eigenvector $\langle u \ts |$ is normalised
such that $\langle u \ts | \ts v \rangle = 1$, which gives
\[
  \langle u\ts\rvert \, = \, 
  \frac{\lambda^d - \lambda}{2 \lambda^d - (d+1)\lambda + (d-1)}
  \biggl(\lambda,\sum_{j=0}^{d-2}\lambda^{-j},
  \sum_{j=0}^{d-3}\lambda^{-j},\ldots,
  1+\lambda^{-1},1  \biggr).
\]
Here, the normalisation prefactor was simplified via the algebraic
relation for $\lambda$ from $p^{}_{d} (\lambda)=0$, which in particular gives
$\lambda^d (\lambda - 1) = \lambda^d - 1$. Note that the vector on the
right-hand side is a canonical choice for the natural interval
lengths, which all lie in $\ZZ[\lambda]$.  The shortest interval then
has length $1$, and it is straightforward to show that no proper,
$\lambda$-invariant submodule of $\ZZ[\lambda]$ contains all control
point positions.  Here, the density of the resulting point set $\vL$
is 
\[
   \dens(\vL) \, = \,  
   \frac{\lambda^d - \lambda}{2 \lambda^d - (d+1)\lambda + (d-1)}\ts ,
\]
with $\lim_{d\to\infty}\dens(\vL)=\frac{1}{2}$. 

When working with the $\ZZ$-module $L=\ZZ[\lambda]$, one can define
the dual module $L^{\circledast}$ with respect to the quadratic form
$\tr (xy)$ as explained in Remark~\ref{rem:spec}, namely
\begin{equation}\label{eq:def-Lmod}
  L^{\circledast} \, = \, \bigl\{ y\in\QQ(\lambda) : \tr(xy)\in\ZZ 
  \text{ for all } x\in L  \bigr\} .
\end{equation}
For our family, one finds $L^{\circledast}=\vartheta\ts L$ with
\[
   \vartheta \, = \, \biggl(d\,\lambda^{d-1} - 
   \sum_{m=0}^{d-2} (m+1)\lambda^{m}\biggr)^{\! -1} \, \in\:
   \myfrac{1}{\Delta_{d}} \, \ZZ[\lambda]\ts .
\]

We are now set to look at some special cases in more detail.

\subsection{The Fibonacci tiling}\label{sec:Pisa-Fib}

For $\varrho^{}_{2} = (ab,a)$, which we write with the binary alphabet
$\cA = \{ a,b \}$ for simplicity, the inflation tiling with interval
lengths $\tau=\frac{1}{2}(1+\sqrt{5}\,)$ for $a$ and $1$ for $b$ is
well studied; see \cite[Sec.~9.4.1]{TAO} and references therein. For
the standard fixed point of the square of $\varrho^{}_{2}$, with
central seed $a|a$, one obtains the windows
$W_{\nts a}=(\tau-2,\tau-1]$ and $W_b=(-1,\tau-2]$, compare
\cite[Ex.~7.3]{TAO}, and can calculate their Fourier transforms
immediately. With $\sinc (z) = \frac{\sin (z)}{z}$, they read
\[
  \widecheck{1^{}_{W^{}_{\nts a}}} (y)
  \, = \, \ee^{\pi \ii y (2 \tau -3)} \sinc (\pi y)
  \quad \text{and} \quad \widecheck{1^{}_{W_b}}
  (y) \, = \, \myfrac{\ee^{\pi \ii y (\tau - 3)}}{\tau} \ts \sinc
  \Bigl(\myfrac{\pi y}{\tau} \Bigr) .
\]
Here, it does not matter whether we take open, half-open or closed
intervals, as their characteristic functions are equal as
$L^{1} $-functions. Consequently, this detail is spectrally invisible.

The internal Fourier matrix and cocycle for this example read 
\[
    \sB(y) \, = \, \begin{pmatrix} 1 & 1 \\ 
    \ee^{2\pi\ii\sigma y} & 0 \end{pmatrix} \quad\text{and}\quad
    \sB^{(n)}(y) \, = \, \sB(y) \sB(\sigma y) \cdots 
    \sB(\sigma^{n-1}y)\ts , 
\]
with\footnote{Here, $\sigma$ is a number which should not be confused
  with the Galois isomorphisms from Section~\ref{sec:Minkowski}.}
$\sigma=\tau^{\star}=1-\tau$, so $\lvert \sigma \rvert = - \sigma$.
We find the relation
\[
   c^{}_{b}(y) \,= \, \lvert \sigma \rvert \ts
       \ee^{2\pi\ii\sigma y} c^{}_{a}(y)
\]
expressing $c^{}_{b}$ in terms of $c^{}_{a}$, while the latter is
obtained as the limit
\[
    c^{}_{a}(y) \, = \lim_{n\to\infty} q^{}_{n} (y)\ts ,
\]
where the trigonometric polynomials $q^{}_{n}$ are recursively defined
by
\[
    q^{}_{n+1}(y) \,=\, \lvert \sigma \rvert \, q^{}_{n} (\sigma y) +
    \sigma^{2} \ee^{2\pi\ii\sigma^2 y} q^{}_{n-1}(\sigma^2 y)\ts ,
\]
with initial conditions $q^{}_{1} = q^{}_{0} = \lvert \sigma \rvert$.
From here, it is not difficult to check that
\[
    c^{}_{a} (y) \, = \, \lvert \sigma \rvert \,
    \widecheck{1^{}_{W^{}_{\nts a}}} (y)
    \quad \text{and} \quad
    c^{}_{b} (y) \, = \, \lvert \sigma \rvert \,
    \widecheck{1^{}_{W_{b}}} (y)
\]
as it must.  The convergence of the recursive formula for $c^{}_{a}$
is exponentially fast. Though there is no need for this alternative
approach in this case, it provides a consistency check and some
additional insight into the recursive structure of the spectrum.

\subsection{(Twisted) Tribonacci}\label{sec:Pisa-Tri}

Here, we compare two different substitution rules for $d=3$, which
share the same substitution matrix, $M=M_{3}$.  These are the
Tribonacci substitution $\varrho^{}_{3} \defeq (a b, a c, a)$ and its
twisted counterpart, $\varrho^{\ts \prime}_{3} \defeq (b a, a c, a)$,
with the alphabet $\{ a,b,c \}$.  Further permutations of letter
positions do not define new hulls, as they are conjugate to one of
these two.  Both lead to inflation systems with fractal windows
in their model set description, as they must due to a result by
Pleasants \cite[Prop.~2.35]{Peter}, but the twisted version is 
more tortuous; see Figure~\ref{fig:triwin} below, and compare
\cite[Figs.~7.5 and~7.8]{PFBook}, where a different coordinate system
is used. The fundamental group of the windows in the twisted case is
huge, while the windows of the untwisted case are still simply
connected, as in other examples such as the inflation
tiling that underlies the Kolakoski-$(3,1)$ sequence \cite{BS}.

The field $\QQ (\lambda)$ is cubic. For ease of notation, we define
$\kappa^{}_{\pm}=\bigl(19\pm 3\sqrt{33}\,\bigr)^{\frac{1}{3}}$. With
this, we find that the PF eigenvalue is
\[
  \lambda \,=\, \tfrac{1}{3}\bigl(1 +
  \kappa^{}_{+}+\kappa^{}_{-}\bigr)
  \, \approx \, 1.839{\ts}287 .
\]
The characteristic polynomial is cubic, $p(x)=x^3 - x^2 - x - 1$, with
discriminant $\Delta=-44$. The remaining two eigenvalues form a
complex conjugate pair $\alpha,\overline{\alpha}$ with
$\lvert\alpha\rvert^2=\lambda^{-1}=\lambda^2-\lambda-1$, where we
assume $\alpha$ to be the one with positive imaginary part. One also
has $\lambda^{-2} = \lambda (2-\lambda)$.  Further, one finds
$\real(\alpha)=(1-\lambda)/2$ and $\real(\alpha^2)=(3-\lambda^2)/2$,
while
$\imag(\alpha)=\frac{1}{2\sqrt{3}}\bigl(\kappa^{}_{+}-\kappa^{}_{-}\bigr)$
and $\imag(\alpha^2)=(1-\lambda)\imag(\alpha)$. From the discriminant
and Vieta's theorem, one also gets
\[
  \imag (\alpha) \, = \, \frac{\sqrt{11}}
       {3 \lambda^2 - 2 \lambda - 1}
       \, = \, \frac{\sqrt{11}}{22} \ts
       \bigl(-4 \lambda^2 + 9 \lambda + 1\bigr) \ts . 
\]

The natural tile lengths can be chosen as
$(\lambda,\lambda^2-\lambda,1)\approx (1.839,1.544,1)$, which means
that all control point positions lie in the rank-$3$ $\ZZ$-module
$L = \langle 1, \lambda, \lambda^2 \rangle^{}_{\ZZ}$, but in no proper
submodule. The lattice for the CPS, $\cL$, is obtained from the
Minkowski embedding of $L$ into $3$-space.  A canonical choice for the
basis matrix of $\cL$ and its dual, $\cL^*$, is then given by
\[
    \cB\, = \, \begin{pmatrix}
  1 & \lambda & \lambda^2 \\
  1 & \real(\alpha) & \real(\alpha^2) \\
  0 & \imag(\alpha) & \imag(\alpha^2)
  \end{pmatrix} \quad\text{and}\quad
  \cB^{*} \, = \, \frac{\imag (\alpha)}{\sqrt{11}}
   \begin{pmatrix}
   \lambda^2-\lambda-1 & \lambda - 1  & 1  \\
   2 \lambda^2 - \lambda & 1-\lambda & -1 \\
   \frac{3 \lambda - \lambda^2}{2 \imag (\alpha)}&  
   \frac{3 (\lambda^2-1)}{2 \imag (\alpha)} &
   \frac{1-3 \lambda}{2 \imag (\alpha)}
\end{pmatrix}
\]
with
$\det(\cB)=\imag(\alpha)\ts (3\lambda^2 - 2 \lambda -1) =
\sqrt{11}$. From the first row of $\cB^{*}$, one can now extract the
Fourier module in our setting from an independent calculation, which
gives
\[
   L^{\circledast} \, = \, \vartheta\, \langle 
   \lambda^2-\lambda-1, \lambda - 1, 1 \rangle^{}_{\ZZ}
   \, = \, \vartheta \ts L \ts ,
\]
with $\vartheta = (3 \lambda^2 - 2 \lambda - 1)^{-1}$, in agreement
with our general formula \eqref{eq:def-Lmod}. The Abelian group
$L^{\circledast}$ is also the dynamical spectrum (in additive
notation) of our systems; compare Remark~\ref{rem:spec}.  In fact,
Tribonacci and twisted Tribonacci are metrically isomorphic by the
Halmos--von Neumann theorem, but have rather different
eigenfunctions. Also, they are obviously \emph{not} mutually locally
derivable (MLD) from one another; see \cite[Sec.~5.2]{TAO} for
background. Moreover, they are not topologically conjugate either, as
they can be distinguished via invariants of gauge-theoretic origin
\cite{Franz}, or by dimension arguments as follows.

\begin{figure}
\centerline{\includegraphics[width=\textwidth]{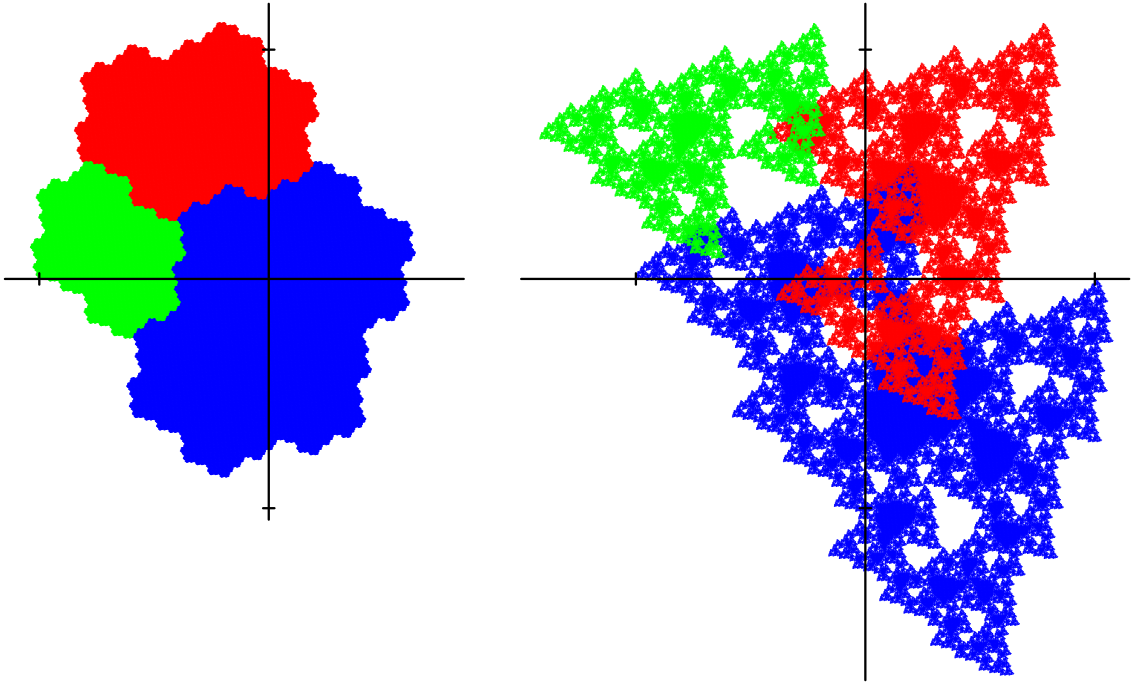}}
\caption{\label{fig:triwin}Rauzy fractals for the Tribonacci inflation
  (left panel) and its twisted sibling (right panel), shown at the
  same scale. They are the windows for the points of type $a$ (blue),
  $b$ (red) and $c$ (green). The coordinate axes are those emerging
  from the Minkowski embedding, with ticks indicating unit distances.}
\end{figure}

\begin{remark}
  The Hausdorff dimension of the fractal boundary of the Tribonacci
  windows is known; compare \cite{Ali} as well as
  \cite[Ex.~4.2]{FFIW}. It can be calculated as a similarity
  dimension, which is the real solution
  $s^{}_{\mathrm{H}}$ to the equation
  $\lvert\alpha\rvert^{4s^{}_{\mathrm{H}}}+
  2\ts\lvert\alpha\rvert^{3s^{}_{\mathrm{H}}}=1$. This gives
\[
  s^{}_{\mathrm{H}} \, = \,  2\,\frac{\log(b)}{\log (\lambda)}\, 
  \approx \, 1.093\ts 364\ts ,
\]
where $b$ is the positive real root of $x^4-2\ts x-1$.
  
Likewise, for twisted Tribonacci, the Hausdorff dimension of the
window boundary is given by \cite[Ex.~4.3]{FFIW}
\[
    s^{}_{\mathrm{H}} \, = \,  2\,\frac{\log(b)}{\log (\lambda)}\, 
   \approx \, 1.791\ts 903\ts ,
\]
where $b$ now is the positive real root of $x^6-x^5-x^4-x^2+x-1$, as
one derives from the corresponding graph-directed IFS for the
boundary; see also \cite[Sec.~6.9]{Bernd}.

The much larger Hausdorff dimension for the twisted case corresponds
to a slower decay of the Fourier transform; see \cite[App.~B]{LGJJ}
for an explicit one-dimensional example for which the Fourier
transform shows a power-law decay with exponent $1-d_{B}$, where
$d_{B}$ is the fractal dimension of the boundary, and \cite{GL} for an
interesting asymptotic scaling analysis of such coefficients.  It
would be useful to establish a general result along these lines, which
is of recent interest also with respect to a refinement of the notion
of complexity \cite{FG}. \exend
\end{remark}

If $\sigma^{}_{1} \! : \, \QQ (\lambda) \xrightarrow{\quad} \QQ (\alpha)$
is the field isomorphism induced by $\lambda \mapsto \alpha$, one
determines the $\star$-map of $k\in L^{\circledast}$ as
$k \mapsto k^{\star} \defeq \bigl(\real(\sigma^{}_{1}(k)),
\imag(\sigma^{}_{1}(k))\bigr)^{T}$.  For
$k = k^{}_{p,q,r} \defeq \vartheta \ts (p + q \lambda + r \lambda^2)$,
this gives
\[
  k^{\star}_{p,q,r} \, = \, \begin{pmatrix}
    \frac{1}{44} \bigl( (-p+ 4 q + 17 r) -
       (9 p - 3 q + r) \lambda +
       (4 p - 5 q - 2 r) \lambda^2 \bigr) \\
    \frac{1}{4 \sqrt{11}} \bigl( (-p + 2 q + r) +
       3 (p+q+r) \lambda -
       (3 q + 2 r) \lambda^2 \bigr) \end{pmatrix},
\]
where the integers $p,q,r$ are known as the \emph{Miller indices} of
the corresponding Bragg peak in crystallography.

\begin{figure}
\centerline{\includegraphics[width=0.85\textwidth]{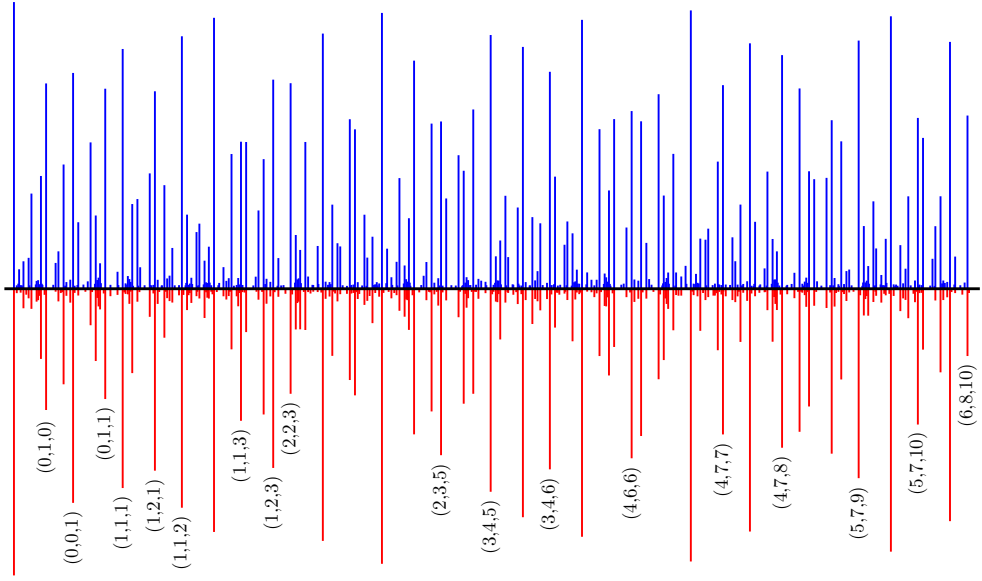}}
\caption{\label{fig:tribodiff}Diffraction intensities (Bragg peaks)
  for the Tribonacci point set (upper part, blue) and for its twisted
  counterpart (lower part, red). Displayed are the relevant peaks for
  $k\in L^{\circledast}\cap [0,10]$, with the intensity represented by
  the length of the line. The left-most peak is located at the origin
  and has height $\dens(\vL)^2$, where
  $\dens(\vL)=\frac{1}{22}(5+\lambda+2\lambda^2)\approx 0.618\ts
  420$. Selected peaks are labelled by their Miller index triples.}
\end{figure}

In Figure~\ref{fig:tribodiff}, we compare the peaks of the pure point
diffraction measure for the Tribonacci point set and its twisted
sibling. The support is the same, but the intensities show
characteristic differences. The latter are calculated as
\[
    I(p,q,r) \, = \, 
    \Bigl(\myfrac{5+\lambda+2\lambda^2}{22}\Bigr)^{\! 2} \,
    \left| \big\langle 1 \ts | \ts C(k^{\star}_{p,q,r}) \ts |\ts v
    \big\rangle\right|^2
\]
with the appropriate matrix function $C$ for the two cases. The peaks
of the twisted case are often smaller than their untwisted
counterparts. Note also that an approximation of the diffraction
measure by exponential sums of large patches suffers from slow
convergence, in particular for the twisted version, as was previously
observed and discussed for the plastic number PV inflation
\cite{ICQ}. This reference also contains an illustration of the full
Fourier transform of the plastic number Rauzy fractal, which shows
similar features as our case at hand.

\subsection{The quartic case}

Let us briefly consider $\varrho^{}_{4}=(01,02,03,0)$ on
$\cA=\{0,1,2,3\}$, where $\QQ (\lambda)$ is a quartic field.  Beyond
the PF eigenvalue $\lambda\approx 1.927\ts 562$, $M$ has one real root
$\mu$, with $\mu\approx -0.774\ts 804$, and a complex conjugate pair
$\alpha,\overline{\alpha}$, with
$\alpha\approx -0.076\ts 379 + 0.814\ts 704\ts\ts \ii$.  For the
natural choice of interval lengths,
$(\lambda,\lambda^2-\lambda,\lambda^3-\lambda^2-\lambda,1)$, the
Fourier module becomes
\[
   L^{\circledast} \, = \, \vartheta \,
   \langle \lambda^3-\lambda^2-\lambda-1,  \lambda^2-\lambda-1,
   \lambda-1, 1 \rangle^{}_{\ZZ} \, = \, 
   \vartheta\, \ZZ[\lambda]\ts ,
\]
where 
\[
  \vartheta\, = \, \bigl(\lambda^3-3\lambda^2-2\lambda-1\bigr)^{-1}
  \, = \, \myfrac{1}{563}
  \bigl(10+157\lambda-103\lambda^2+16\lambda^3\bigr).
\]
This follows from \eqref{eq:def-Lmod} and can be verified via the
quadratic form $\tr(xy)$, observing $\tr(1)=4$ and
$\tr(\lambda^m)=2^m-1$ for $m\in\{1,2,3\}$.  In analogy to before, we
parametrise $k\in L^{\circledast}$ by a quadruple $(p,q,r,s)$ of
Miller indices.

The internal Fourier matrix reads
\[
   \sB(y) \, = \, 
   \begin{pmatrix} 1 & 1 & 1 & \, 1\, \\
    e(y) & 0 & 0 & 0 \\
   0 & \! e(y)\! & 0 & 0 \\
   0 & 0 & e(y) & 0 
   \end{pmatrix}
\]
with
$e(y)\defeq \exp\bigl(2\pi\ii (\mu\ts\ts y^{}_{1} + \real(\alpha)\ts
y^{}_{2} + \imag(\alpha)\ts y^{}_{3})\bigr)$. A calculation analogous
to our previous ones leads to the diffraction measure as illustrated
in Figure~\ref{fig:case4}. Let us briefly mention that, using the
methods from \cite[Cor.~4.118 and Prop.~4.122]{Bernd}, one can derive
an upper bound of $2.327$ for the Hausdorff dimension of the window
boundaries \cite{BerndPC}. It is no problem to twist $\varrho^{}_{4}$,
as we did for the Tribonacci case, but we leave further details to
the interested reader.

\begin{figure}
\centerline{\includegraphics[width=0.85\textwidth]{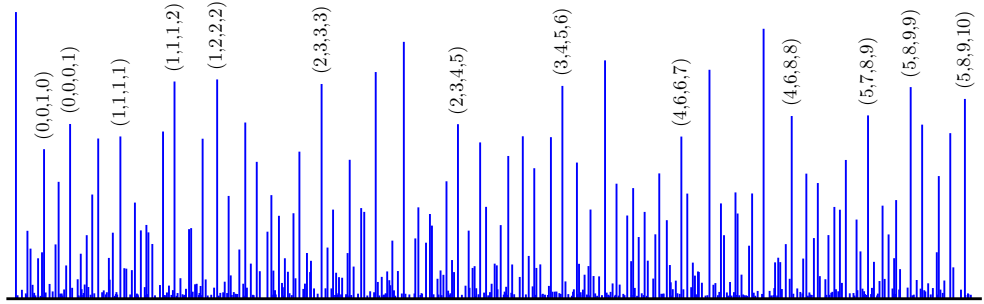}}
\caption{\label{fig:case4}Illustration of the pure point diffraction
  spectrum for the $d=4$ Pisa inflation, for
  $k\in L^{\circledast}\cap [0,10]$. The left-most peak is located at
  $0$ and has height $\dens(\vL)^2$, with
  $\dens(\vL)=\frac{1}{563}\ts
  (86-\lambda+15\lambda^2+25\lambda^3)\approx 0.566\ts 343$.  Selected
  peaks are labelled by their Miller index quadruples.}
\end{figure}

\section{Twisted extensions of Fibonacci chains with mixed
  spectrum}\label{sec:twisted}

Let us close with a simple system with mixed spectrum. It is based on
the idea, taken from \cite{BG}, of a twisted extension of
$\varrho^{}_{\mathrm{F}} = (ab,a)$, which is $\varrho^{}_{2}$ from
Section~\ref{sec:Pisa-Fib}, with a \emph{bar swap symmetry}. As such,
it works with the extended alphabet
$\cA = \big\{ a, \bar{a}, b , \bar{b} \big\}$, where we consider
\[
  \varrho \, = \, (ab,\bar{a}\bar{b},\bar{a},a) \ts .
\]
The natural interval lengths are those of the Fibonacci tiling, so
$\tau$ for $a$ and $\bar{a}$, as well as $1$ for $b$ and $\bar{b}$.
Also, by identifying $a$ with $\bar{a}$ and $b$ with $\bar{b}$, one
sees that the system possesses the Fibonacci tiling as a topological
factor, where the factor map is $2:1$ almost everywhere, but not
everywhere \cite{Franz}.  Consequently, we have a non-trivial point
spectrum, together with a continuous component. The latter, by an
application of the renormalisation methods from
\cite{BGM,Neil-thesis,Neil}, must be singular continuous.

The substitution matrix of $\varrho$ has spectrum
$\{ \tau, 1-\tau, \frac{1}{2} (1 \pm \ii \sqrt{3} \,) \}$, and a
reducible characteristic polynomial. Only the factor with $\tau$ as a
root is relevant, and one checks that the same embedding as for the
Fibonacci tiling can be used. Here, the embedding method produces
covering supersets, where the contractive IFS on $(\cK \RR)^4$ reads
\begin{align*}
  W^{}_{\nts a} & \, = \, \sigma W^{}_{\nts a} \cup 
              \sigma W^{}_{\bar{b}} \, , &
              W^{}_{b} & \, = \, \sigma W^{}_{\nts a} + 
              \sigma \ts , \\
  W^{}_{\nts \bar{a}} & \, = \, \sigma W^{}_{\nts \bar{a}} \cup 
              \sigma W^{}_{b} \, , &
              W^{}_{\bar{b}} & \, = \, \sigma W^{}_{\nts \bar{a}} + 
             \sigma \ts ,
\end{align*}
with $\sigma = \tau^{\star}$ as before. The unique solution with
compact subsets of $\RR$ is
\begin{equation}\label{eq:f-win}
\begin{split}
  W^{}_{\nts a} \, & = \, W^{}_{\nts \bar{a}} 
  \, = \, [ -\sigma^2, - \sigma ]
  \, = \, [\tau-2,\tau-1]
  \quad\text{and} \\[1mm]
  W^{}_{b} \, & = \, W^{}_{\bar{b}} \, = \, [ -1, -\sigma^2]
  \, = \, [-1,\tau-2] \ts ,
\end{split}  
\end{equation}
as can easily be verified by direct computation. Here, we are in the
situation that $\mf (y) = 2$ for a.e.\ $y\in [-1,\tau-1]$, and
uniform distribution is preserved both in the individual windows,
by Theorem~\ref{thm:uniform-dist}, and in the total window.

Note that the point sets $\vL_a^{\star}$ and $\vL_{\bar{a}}^{\star}$
are disjoint, but have the same closure, and analogously for
$\vL_b^{\star}$ and $\vL_{\bar{b}}^{\star}$.  The right-hand sides of
the window equations are measure-disjoint by Lemma~\ref{lem:disjoint},
which means that the cocycle approach can be applied, with the window
covering degree being $\mc =2$.  Since uniform distribution is
satisfied here by Theorem~\ref{thm:uniform-dist}, the FB
coefficients from \eqref{eq:FB} can be calculated by means of
\eqref{eq:genampli}.  For weights $h^{}_{\alpha}\in\CC$ with
$\alpha\in\{a,\bar{a},b,\bar{b}\}$, the pure point part of the
diffraction reads
\[
  (\widehat{\gamma})^{}_{\mathsf{pp}}\, =
  \sum_{k\in L^{\circledast}} \Big\lvert
  {\textstyle \sum\limits_{\alpha}}\, h^{}_{\alpha} \ts A^{}_{\alpha}(k)
  \Big\rvert^{2} \, \delta^{}_{k} \ts ,
\]
with the additional part of the diffraction measure being singular
continuous.  \smallskip

As was noticed by G\"{a}hler \cite{Franz}, one can employ a partial
return word coding to arrive at another inflation which defines a
tiling system that is MLD with the above. Concretely, consider the
alphabet $\{ A,B,C,D \}$ and the inflation
$\varrho^{\ts\prime} = (AB,D,C{\nts}A,C)$. Here, $A$ and $B$
correspond to $a$ and $b$, while $C$ replaces $\bar{a}\bar{b}$ and $D$
replaces each $\bar{a}$ that is not followed by a $\bar{b}$. This gives
the substitution matrix
\[
  \begin{pmatrix}
    1 & 0 & 1 & 0 \\ 1 & 0 & 0 & 0 \\
    0 & 0 & 1 & 1 \\ 0 & 1 & 0 & 0
    \end{pmatrix}
\]
with the same eigenvalues as above.  The natural interval lengths are
$(\tau, 1, \tau+1, \tau)$, in agreement with the local derivation rule
just stated.

The resulting window equations read
\[
  W^{}_{A} \, = \, \sigma W^{}_{A} \cup ( \sigma W^{}_{C}
  + \sigma^2 ) \, , \quad W^{}_{B} \, = \,
    \sigma W^{}_{A} \nts + \sigma
  \, , \quad W^{}_{C} \, = \, \sigma W^{}_{C} \cup \sigma W^{}_{D}
  \, ,\quad W^{}_{D} \, = \, \sigma W^{}_{B} \ts ,
\]
which constitute a contractive IFS on $(\cK \RR)^4$
with unique solution
\[
  W^{}_{A} \, = \, [\tau \nts - \nts 2,
      \tau \nts - \nts 1] \, , \quad
      W^{}_{B} \, = \, [\ts -1, \tau \nts - \nts 2] \, , \quad
  W^{}_{C} \, = \, [\tau \nts - \nts 2,
       2 \tau \nts - \nts 3] \, , \quad
  W^{}_{D} \, = \, [ 2\tau \nts - \nts 3,
          \tau \nts - \nts 1] \ts .
\]
The total window is $[-1, \tau-1]$ as in the twisted Fibonacci
example, but the window function $\mf$ now is a step function
as induced by
\begin{equation}\label{eq:winpic}
  \raisebox{-25pt}{\includegraphics[width=0.4\textwidth]{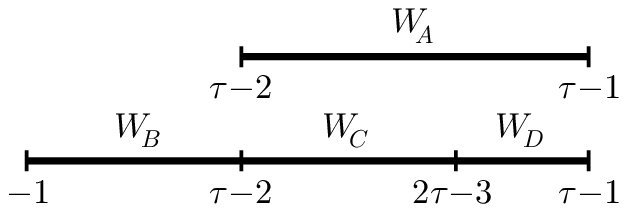}}
\end{equation}

As was further analysed by G\"{a}hler \cite{Franz}, there is also a
maximal topological pure point factor such that the factor map is
\mbox{{\ts}2{\,}:{\ts}1\ts} everywhere.  Using the alphabet
$\{ 0,1,2,3 \}$, this maximal pure point factor is given by the
inflation rule
\[
  \widetilde{\varrho} \, = \, (12,13,1,0) \ts ,
\]
where $0$ and $1$ stand for intervals of length $\tau$, while those of
type $2$ and $3$ have unit length. The factor map can most easily be
given as a block map, where words of length $2$ at position $n$ are
mapped to an element of the new alphabet at the same position, namely
\begin{equation}\label{eq:block}
  aa, \bar{a} \bar{a} \, \mapsto \, 0 \, , \quad
  ab, \bar{a}\bar{b}, a \bar{a}, \bar{a} a
  \, \mapsto \, 1 \, , \quad
  ba, \bar{b}\bar{a} \, \mapsto \, 2 \, , \quad
  b \bar{a}, \bar{b} a \, \mapsto \, 3 \ts ,
\end{equation}
and correspondingly for the tilings, where the resulting mapping
is called a \emph{local derivation rule}; see \cite[Sec.~5.2]{TAO}
for details.

\begin{figure}
\centerline{\includegraphics[width=0.8\textwidth]{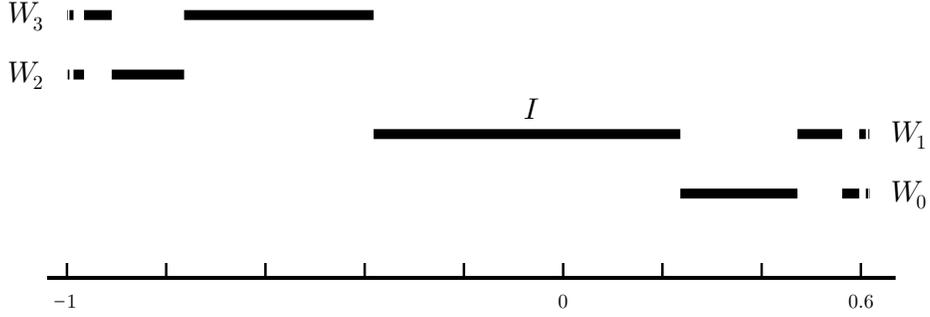}}
\caption{\label{fig:fib4win}Illustration of the four windows for the
  primitive inflation rule
  \mbox{$\widetilde{\varrho}=(12,13,1,0)$}. Note that $W_{0}\cup
  W_{1}=[-\sigma^2,-\sigma]$ and $W_{2}\cup W_{3}=[-1,-\sigma^2]$,
  while $I=[-\sigma^2,-\sigma^3]$.}
\end{figure}

Conversely, one proceeds in two steps. First, any given sequence from
the (symbolic) hull of $\widetilde{\varrho}$ is mapped to a sequence
in $\{a,b\}^{\ZZ}$ by $0,1\mapsto a$ and $2,3\mapsto b$. In the second
step, choose one position and decide whether to place a bar on the
letter or not. Then, the bar status of the two neighbouring symbols is
uniquely determined from the original block map \eqref{eq:block}, read
backwards. Inductively, this fixes the entire sequence. Since the only
choice was the initial bar, this shows that the block map
\eqref{eq:block} is globally \mbox{\ts2{\,}:{\ts}1\ts}. Once again,
this block map transfers to a local derivation rule for the
corresponding tilings.

The new inflation rule $\widetilde{\varrho}$ leads to a regular model
set, with window equations
\[
  W^{}_{0} \, = \, \sigma W^{}_{3} \, , \quad
  W^{}_{1} \, = \, \sigma W^{}_{0} \cup \sigma W^{}_{1}
  \cup \sigma W^{}_{2} \, , \quad
  W^{}_{2} \, = \, \sigma W^{}_{0} + \sigma \, , \quad
  W^{}_{3} \, = \, \sigma W^{}_{1} + \sigma \ts .
\]
Due to the factor map, we immediately know that
$W^{}_{0} \cup W^{}_{1} = W^{}_{\nts a}$ and
$W^{}_{2} \cup W^{}_{3} = W^{}_{b}$ with $W^{}_{\nts a}$ and
$W^{}_{b}$ from \eqref{eq:f-win}.  The unique solution can be
determined by first observing that each $W_{i}$ with $i\ne 1$ can be
expressed in terms of $W^{}_{1}$. This gives a rescaling equation for
$W^{}_{1}$ alone, namely
\[
  W^{}_{1} \, = \, \sigma W^{}_{1} \cup
  \bigl( \sigma^3 W^{}_{1} + \sigma^3 \bigr) \cup
  \bigl( \sigma^4 W^{}_{1} + \sigma^4 + \sigma^2 \bigr) 
  \, = \, I \cup g(W^{}_{1}) \ts ,
\]
where $I = (W^{}_{2} \cup W^{}_{3}) - \sigma = [-\sigma^2, -\sigma^3]$
and $g (x) = \sigma^4 x + \sigma^4 + \sigma^2$. 

This leads to $W^{}_{1} = I \cup g(I) \cup g(g(I)) \cup \ldots$ which
results in the formula
\begin{equation}\label{eq:W1}
  W^{}_{1} \, = \bigcup_{n\geqslant 0}
  \bigl( \sigma^{4n} [-\sigma^2, -\sigma^3] +
  \sigma (\sigma^{4n} -1) \bigr)  ,
\end{equation}
while the other windows follow from here via affine mappings. All four
windows are illustrated in Figure~\ref{fig:fib4win}, each comprising
countably many disjoint intervals.

The explicit expression for $W^{}_{1}$ in \eqref{eq:W1} leads to the
(inverse) Fourier transform of its characteristic function in the form
\begin{equation}\label{eq:win-sum}
   f^{}_{1}(y) \, = \, \widecheck{1^{}_{W_{1}}}(y) 
   \, = \, - \sum_{n=0}^{\infty} 
   \sigma^{4n+1} \ee^{-\pi \ii (2\sigma + 2\sigma^{4n} + \sigma^{4n+1}) y} 
   \sinc(\pi\sigma^{4n+1}y)
\end{equation}
with $f^{}_{1}(0)=\tau/\sqrt{5}=\frac{\tau+2}{5}$, which is the total
length of the window $W^{}_{1}$; see Figure~\ref{fig:sincpic} for a
comparison of $\lvert f^{}_{1}\rvert$ with the function
$\big\lvert \frac{\tau+2}{5}\ts \sinc\bigl(\frac{\tau+2}{5}\ts\pi
y\bigr) \big\rvert$.

\begin{figure}
\centerline{\includegraphics[width=0.8\textwidth]{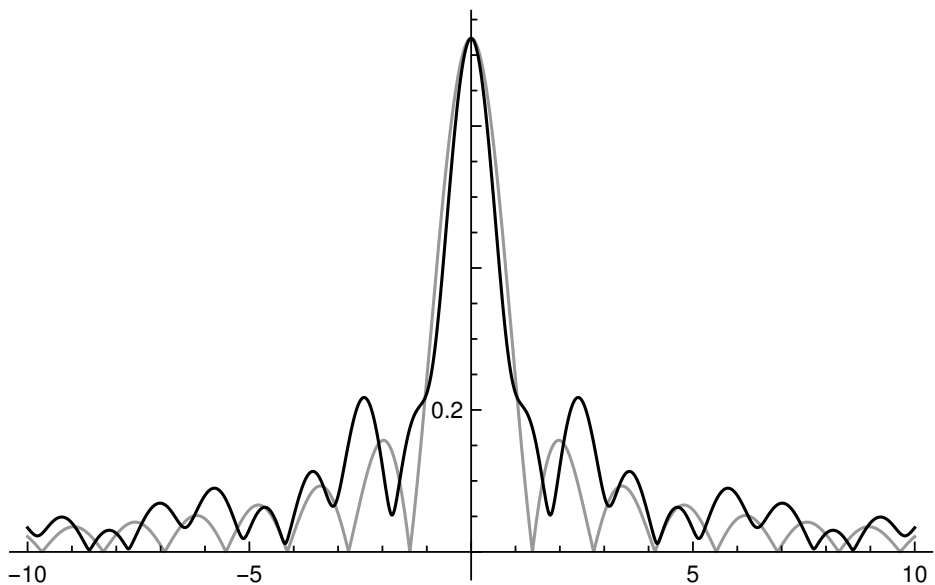}}
\caption{\label{fig:sincpic}Illustration of
  $\,\lvert\widecheck{1^{}_{W_{1}}}(y)\rvert\ts$ (black curve) in
  comparison with the modulus of the Fourier transform of an interval
  of length $\frac{\tau+2}{5}$ (grey curve), which is the value at $0$
  for both functions.}
\end{figure}

For the cocycle approach, we first note that the internal Fourier
matrix reads
\[
      \sB(y) \, = \, \begin{pmatrix}
      0 & 0 & 0 & \, 1 \, \\
      1 & 1 & 1 & 0 \\
      \!\ee^{2\pi\ii\sigma y}\! & 0 & 0 & 0 \\
      0 & \!\ee^{2\pi\ii\sigma y}\! & 0 & 0 
      \end{pmatrix}
\]
with $\sB(0)=M$ as usual. The frequency-normalised right PF
eigenvector is
\[
   |\ts v\rangle \, = \, \tfrac{1}{5} (-1 - 3\sigma, 1 - 2\sigma, 
   3 + 4 \sigma, 2 + \sigma )^T \, \approx\, 
   (0.171, 0.447, 0.106, 0.276)^T,
\]
where one has $\vol (W^{}_{i}) = \tau \ts v^{}_{i}$ for the total
window lengths.  With
$C(y)=\lim_{n\to\infty}\lvert\sigma\rvert^n \sB^{(n)}(y)$, one gets
$f^{}_{1}(y) = \tau \, \langle \ts 0,1,0,0\ts\ts | \ts C(y) \ts |\ts
v\rangle$, where the convergence of the underlying matrix product is
exponentially fast. Here, one can then study the rate of convergence
in comparison to the alternative formula in \eqref{eq:win-sum}.

\section{Outlook}

It is possible to extend our approach to inflation tilings in higher
dimensions, if the inflation multiplier is a PV unit. In fact, this is
needed and useful when dealing with direct product variations (DPV) 
as considered in \cite{Nat-primer,Natalie,BFG}.

An extension to the non-unit case is also possible, but requires a
larger machinery from algebraic number theory, as developed in
\cite{Bernd} for the treatment of the Pisot substitution conjecture in
the general non-unit case. This is work in progress.

Finally, also $S$-adic type inflations can be covered, provided that
the participating inflation rules are compatible in the sense that
they share the same substitution matrix. Further, there are
applications of our results to the Eberlein decomposition for Dirac
combs of primitive inflation systems \cite{BaSt} and consequences for
the spectral theory of regular sequences \cite{CEM}. \bigskip

\section*{Acknowledgements}

It is a pleasure to thank N.P.~Frank, F.~G\"{a}hler, N.~Ma\~{n}ibo,
B.~Sing and N.~Strungaru for helpful discussions, and an anonymous
referee for several thoughtful suggestions that helped to improve the
presentation. Our work was supported by the German Research Foundation
(DFG), within the CRC 1283 at Bielefeld University, and by EPSRC
through grant EP/S010335/1.  

\end{document}